\documentclass[10pt]{article}
\usepackage{graphicx}
\graphicspath{ {images/} }
\usepackage{amsthm}
\usepackage[english]{babel}
\usepackage{graphicx,subfigure}
\usepackage{amsmath,amssymb}
\usepackage{appendix}
\usepackage{color}
\usepackage{xcolor}
\usepackage{soul}
\usepackage{amsmath,amsthm,amsfonts,amssymb,bm}
\usepackage[light]{antpolt}    \usepackage[T1]{fontenc} 
\usepackage{mathrsfs,amssymb,bm,enumerate}
\usepackage{multirow}

\usepackage{ulem}

\usepackage[pagewise,mathlines,switch]{lineno}

\usepackage{amsbsy}
\DeclareMathAlphabet{\mathpzc}{OT1}{pzc}{m}{it}
\newcommand{\vertiii}[1]{{\left\vert\kern-0.25ex\left\vert\kern-0.25ex\left\vert #1 
		\right\vert\kern-0.25ex\right\vert\kern-0.25ex\right\vert}}
\usepackage{epsfig}

\numberwithin{equation}{section}
\newtheorem{theorem}{\qquad Theorem}[section]
\newtheorem{lemma}[theorem]{\qquad Lemma}

\newtheorem{remx}[theorem]{Remark}
\newenvironment{remark}
{\pushQED{\qed}\remx}
{\popQED\endremx}

\newcommand{\dd}{\;{\rm d}}

\newcommand{\scal}[1]{\left\langle #1 \right\rangle}

\newcommand{\nx}{\mathscr{X}}

\newcommand{\fff}{\mathfrak{F}}

\newcommand{\norm}[1]{\left\|   #1   \right\|}

\def\ned{\partial_x^{-1}}

\allowdisplaybreaks

\newcommand{\al}{\alpha}
\newcommand{\rr}{\mathbb{R}}
\newcommand{\lt}{ {L^2} }

\DeclareMathOperator{\sech}{sech}

\title{Mathematical properties of  Klein-Gordon-Boussinesq systems 
	}
\author{A. Dur\' an$^1$\thanks{Corresponding author} \and A. Esfahani$^2$ \and G. M. Muslu$^3$}

\date{
	$^1$Applied Mathematics Department, University of Valladolid, Valladolid, Spain \\ \texttt{{angeldm}@uva.es}\\%
$^2$Department of Mathematics, Nazarbayev University, Astana, Kazakhstan \\ \texttt{amin.esfahani@nu.edu.kz}\\%
$^3$Istanbul Technical University, Department of Mathematics, Istanbul, Turkey\\
\texttt{gulcin@itu.edu.tr} \\[2ex]%
}

\begin{document}
	\maketitle
	\begin{abstract}
	The Klein-Gordon-Boussinesq (KGB) system is proposed in the literature as a model problem to study the validity of approximations in the long wave limit provided by simpler equations such as KdV, nonlinear Schr\"{o}dinger or Whitham equations. In this paper, the KGB system is analyzed as a mathematical model in three specific points. The first one concerns well-posedness of the initial-value problem with the study of local existence and uniqueness of solution and the conditions under which the local solution is global or blows up at finite time. The second point is focused on traveling wave solutions of the KGB system. The existence of different types of solitary waves is derived from two classical approaches, while from their numerical generation several properties of the solitary wave profiles are studied. In addition, the validity of the KdV approximation is analyzed by computational means and from the corresponding KdV soliton solutions.
	\end{abstract}

	
	
	
	
 
	\section{Introduction}
	In this paper, we study the initial-value problem (ivp)
		\begin{equation}\label{kgb}
		\begin{cases}
			u_{tt}=\al^2u_{xx}+u_{ttxx}+(f_1(u,v))_{xx},\\
			v_{tt}=v_{xx}-v+ f_2(u,v), \qquad\qquad x,t\in\rr,\\
			u(x,0)=u_0,\quad u_t(x,0)=u_1,\\
			v(x,0)=v_0,\quad v_t(x,0)=v_1,
		\end{cases}
	\end{equation}
	where
	
	\begin{equation}\label{kgb1}
		f_1(u,v)=a_{uu}u^2+2a_{uv}uv+a_{vv}v^2,\qquad 
		 f_2(u,v)=b_{uu}u^2+2b_{uv}uv+b_{vv}v^2,
	\end{equation} 	
and the coefficients $\alpha$, $a_{uu}$, $a_{uv}$, $a_{vv}$, $b_{uu}$, $b_{uv}$, $b_{vv}$ are real values, with $\alpha\neq 0$ and $a_{\gamma\beta},b_{\gamma\beta}, \beta,\gamma=u,v$, not all zero. Here, $u$ and $v$ are   real-valued functions.
The set of equations \eqref{kgb}, \eqref{kgb1} is called Klein-Gordon-Boussinesq (KGB) system. The KGB system is used to discuss the validity of the approximation given by equations like KdV, nonlinear Schr\"{o}dinger (NLS) or Whitham equations for systems in periodic media, \cite{dss,Schneider,ChongS,BauerCS}. The term validity is referred to the existence of a version of the approximation equation whose solutions can be compared to those of the KGB system, in the sense that the errors between solutions can be bounded over long time intervals. The situation illustrated by the KGB system is of particular interest in a two-fold way. First, because of the use of the method of energy estimates as standard procedure to control the error, \cite{KirrmannSM}, cannot be applied directly; second, because when the alternative of transforming the system, via normal forms, and applying the method of energy estimates to the transformed problem is used, then the analysis of the error requires some non-resonance conditions on the disperion relations of the corresponding linearized system to validating the approximation. The KGB system is taken as a model since it posseses a Fourier mode representation with common properties with a Bloch wave representation of the water wave problem, cf. \cite{BauerCS}. The previous approach was applied to establish, for particular values of \eqref{kgb1}, the validity of the KdV approximation and the Whitham approximation in \cite{ChongS} and \cite{dss}, respectively, while the KdV approximation in the general case \eqref{kgb1} is investigated in 
\cite{Schneider-2020,BauerCS}. The last reference and \cite{Schneider} study the NLS approximation.
%
%

  The system \eqref{kgb} is reminiscent of two well-known equations. The improved Boussinesq equation
\begin{equation}\label{imbq}
    u_{tt}=\al^2u_{xx}+u_{ttxx}+(u^{2})_{xx},
\end{equation}
introduced in \cite{Bog77} as a modification of the Boussinesq equation, \cite{Bous},
\begin{equation}\label{bq}
    u_{tt}=\al^2u_{xx}+u_{xxxx}+(u^{2})_{xx},
\end{equation}
modelling the bi-directional propagation of nonlinear dispersive long waves in shallow water under gravity effects. The modification is based on the equivalence between the linear dispersion relation of (\ref{imbq}) and \eqref{bq} for long waves. Generalizations of \eqref{imbq} of the form
\begin{equation*}\label{imbq2}
    u_{tt}=\al^2u_{xx}+u_{ttxx}+(f(u))_{xx},
\end{equation*}
for some homogeneous nonlinearities $f$ were introduced in, e.~g., \cite{mak}, to describe the propagation of nonlinear
waves in plasma; \cite{clar} to model the evolution of  longitudinal deformation waves in
elastic rods; or \cite{WazWaz05}, to investigate the existence of compact and noncompact physical structures.

The second equation in \eqref{kgb} belongs to the family of nonlinear Klein-Gordon equations 
\begin{equation}\label{kg}
    u_{tt}=u_{xx}+G'(u),
\end{equation}
for some smooth function $G:\mathbb{R}\rightarrow\mathbb{R}$. As a nonlinear generalization of the wave equation, \eqref{kg} appears in the modelling of many research areas, depending on the type of the nonlinear term $G'$. The applications concern quantum field theory, nonlinear optics, and some phenomena in Biology, such as nerve pulse propagation along neuron membranes and the dynamics of scalar fields. We refer to, e.~g. \cite{Scott,Infeld} and references therein for more information on \eqref{kg} and its particular cases such as the sine-Gordon equation.


We make a brief review of some available theoretical results on \eqref{imbq} and \eqref{kg} that are of interest for the purpose of the present paper.
Equation \eqref{imbq}, with $f(u)=\pm|u|^{p-1}u$, is sometimes referred as the Pochhammer-Chree equation, \cite{liu}. This model was first introduced by Pochhammer, \cite{poch}, and in its complete nonlinear form by Chree \cite{chree}. 
  Liu in \cite{liu} showed local and global well-posedness for \eqref{imbq} $H^s\times H^{s+1}$ with $s \geq1$.  He also showed blow up of negative energy
solutions  but with focusing nonlinearities. This model is also
characterized by the existence of (super-luminal) solitary waves of the form $u(x,t)=Q_{c_s}(x-c_st-x_0)$, with $x_0\in\rr$ and $|c_s|>1$, where $Q_c(r)=(c_s^2-1)^{1/(p-1)}Q(\sqrt{(c_s^2-1)/c_s^2}r)$ and 
\[
Q(r)=\left(\frac{p+1}{2}\sech^2\left(\frac{p-1}{2}r\right)\right)^{\frac{1}{p-1}}.
\]
  The stability or instability of these solitons under the flow of \eqref{imbq} remains an important open question. 

On the other hand, for initial data in the corresponding energy space, local and global well-posedness results  for small solutions of \eqref{kg} are well-known (see for example \cite[ Theorem 6.2.2 and Proposition 6.3.3]{cazhar}). See also \cite{del-1,del-2}. Moreover, stability and instability of standing waves of \eqref{kg} were studied in \cite{shatah-1,shatah-2,ss}.

The purpose of the present work is to analyze several mathematical properties of the KGB system which concern \eqref{kgb}, \eqref{kgb1} as a model and that were, to the best of our knowledge, not considered in the literature yet. The main contributions are the following:
\begin{enumerate}
\item Well-posedness of the ivp \eqref{kgb}, \eqref{kgb1} is analyzed. Existence and uniqueness of solutions, locally in time, are established on suitable Sobolev spaces and, for some cases of the coefficients in \eqref{kgb1}, conditions for global existence or blow-up in finite time are determined. This outlines the contents of Section \ref{lwp-section}.
 \item Special solutions of \eqref{kgb}, \eqref{kgb1} are investigated in Section \ref{solit-sec}. More specifically, the paper is focused on the existence of solitary wave solutions. The classical approaches based on Normal Form Theory, \cite{IK,Champ,ChampS,ChampT}, and Positive Operator Theory, \cite{BenjaminBB1990,BonaCh2002}, are here used to derive the conditions for the existence of solitary waves of three types: Classical Solitary Waves (CSW), with monotone and nonmonotone decay, and Generalized Solitary Waves (GSW). The numerical generation of the solitary-wave profiles is accurately performed by using Petviashvili's method, \cite{Petv1976},  which may include extrapolation techniques  to accelerate the convergence, \cite{sidi1}. The numerical procedure is described in detail in Appendix \ref{appA}.
\item The validity of a long wave KdV approximation for \eqref{kgb}, \eqref{kgb1}, studied in \cite{BauerCS,Schneider-2020} (and in \cite{ChongS} for particular values of the coefficients in \eqref{kgb1})  is considered in section \ref{kdv_app}. The form of the associated KdV equation is derived and the approximation theorem proved in \cite{BauerCS,ChongS} is investigated numerically from the KdV soliton solution.
\end{enumerate}
The following notation will be used throughout the paper. For real $s$, $H^{s}=H^{s}(\mathbb{R})$ stands for the $L^{2}$-based Sobolev space over $\mathbb{R}$, with norm $||\cdot ||_{H^{s}}$, and $\dot{H}^{s}$ denotes the corresponding homogeneous Sobolev space with norm $||\cdot||_{\dot{H}^{s}}$. For $X=X_{1}\times X_{2}$ a Cartesian product of Sobolev spaces, we will consider the norm
$$||h||_{X}=||h_{1}||_{X_{1}}+||h_{2}||_{X_{2}},\quad h=(h_{1},h_{2})\in X.$$· 
In addition, for $T>0, m\geq 0$, $C_{T}^{m}(X)=C^{m}([0,T),X)$ will denote the space of $m$th-order continuously differentiable functions $h:[0,T)\rightarrow X$. The norm in $C_{T}^{0}(X)=C_{T}(X)$ given by
$$||h||_{C_{T}(X)}={\rm sup}_{0\leq t<T}||h(t)||_{X},$$ will also be used.

 \section{Well-posedness}\label{lwp-section}

In this section, we investigate the existence of local solutions of \eqref{kgb}  and find the conditions under which these solutions are global or blow up in finite time.  The first point will be studied through the application of the classical Contraction Mapping Theorem in suitable spaces. A first step will require the analysis of the ivp for the linearized equations, 
\begin{equation}\label{lkgb}
		\begin{cases}
			u_{tt}=\al^2u_{xx}+u_{ttxx},\\
			v_{tt}=v_{xx}-v, \qquad\qquad x,t\in\rr,
		\end{cases}
	\end{equation}
Note that since the dispersion in the first equation \eqref{kgb} is weak, one cannot expect to derive Strichartz estimates for this equation. On the other hand, the second equation involves Klein-Gordon dispersion, and the dispersive estimates associated with the Klein-Gordon group are well-known (see \cite{gv, NO}). Let $t\mapsto U(t), V(t)$ be the corresponding linear groups which, from the Fourier method applied to \eqref{lkgb}, have Fourier symbols
\begin{equation}\label{fourier1}
\widehat{U(t)g}(\xi)=\sin\left(\frac{t|\alpha\xi|}{\sqrt{1+\xi^2}}\right)\frac{\sqrt{1+\xi^2}}{|\alpha\xi|}\widehat{g}(\xi),\quad 
\widehat{V(t)g}(\xi)=\frac{\sin\left(t \sqrt{1+\xi^2} \right)}{\sqrt{1+\xi^2}}\widehat{g}(\xi),\quad \xi\in\mathbb{R},
\end{equation}
are time differentiable and satisfy
\begin{equation}\label{fourier2}
\widehat{\partial_{t}{U}(t)g}(\xi)=\cos\left(\frac{t|\alpha\xi|}{\sqrt{1+\xi^2}}\right)\widehat{g}(\xi),\quad 
\widehat{\partial_{t}{V}(t)g}(\xi)=\cos  \left(t \sqrt{1+\xi^2} \right)\widehat{g}(\xi),\quad \xi\in\mathbb{R},
\end{equation}
(where $\widehat{g}(\xi)$ denotes the Fourier transform of $g\in L^{2}(\mathbb{R})$ at $\xi$), and, for $t>0, r,s\geq 0$, the estimates, \cite{gv, NO,liu}
\begin{eqnarray}\label{fourier3}
&&\|U(t)f\|_{H^r}\leq\|f\|_{H^r},\quad \|\partial_{t}U(t)f\|_{H^r}\leq\|f\|_{H^r},\\
&&\|V(t)f\|_{H^s}\leq\|f\|_{H^{s-1}},\quad \|\partial_{t}V(t)f\|_{H^s}\leq\|f\|_{H^{s}}.
\end{eqnarray}
 \begin{theorem}\label{local}
 	Let   $s,r\geq 0$ satisfying $1/2<r\leq s\leq r+1$, $(u_0,u_1)\in X^{r}:=H^r\times  (H^{r}\cap\dot{H}^{r-1})$, $(v_0,v_1)\in X^{s}:=H^{s}\times H^{s-1}$
 	Then there exist $T>0$, depending only on the norms $||(u_{0},u_{1})||_{X^{r}}, ||(v_{0},v_{1})||_{X^{s}}$, and a unique solution $(u,v)$ of \eqref{kgb} with $(u,u_{t})\in C^{1}_{T}(X^{r}), (v.v_{t})\in C_{T}^{1}(X^{s})$.
 \end{theorem}
\begin{proof}
Using the operators \eqref{fourier1} and Duhamel's principle, \eqref{kgb} can be written in the integral form
\begin{equation}\label{integralform}
	 	\begin{split}
 		u(x,t)&=\partial_tU(t)u_0+U(t)u_1+
 		\int_0^tU(t-t'){ \partial_x^2}{(I-\partial_x^2})^{-1}f_1(u,v)\dd t',\\ 
 		v(x,t)&=\partial_tV(t)v_0+V(t)v_1+
 		\int_0^tV(t-t') f_2(u,v)\dd t'. 
 	\end{split}
 \end{equation}
Let $T>0$ to be specified later. From \eqref{fourier2}, \eqref{fourier3}, the standard estimate
$$||{ \partial_x^2}{(I-\partial_x^2})^{-1}g||_{H^{r}}\lesssim ||g||_{H^{r}},$$ and \eqref{integralform}, it holds that
\begin{eqnarray}
||(u,u_{t})||_{C_{T}(X^{r})}&\lesssim &||(u_{0},u_{1})||_{X^{r}}+T{\sup}_{0\leq t\leq T} ||f_1||_{H^{r}},\label{estimate1}\\
||(v,v_{t})||_{C_{T}(X^{s})}&\lesssim& ||(v_{0},v_{1})||_{X^{s}}+T{\sup}_{0\leq t\leq T} \|f_2\|_{H^{s-1}}.\label{estimate1b}
 	\end{eqnarray}
Following \cite[Theorem 1]{HSS} (see also \cite[Theorem 2.1]{liu}  and   \cite{gv}), it is enough to control the nonlinear terms $f_1$ and $f_2$ in order to apply the fixed point argument for the Contraction Mapping Theorem in $C_{T}(X^{r})\times C_{T}(X^{s})$. Under the hypotheses on $r$ and $s$, $H^{r}$ and $H^{s}$ are Banach algebras and therefore
\begin{equation*}
\|u^2\|_{H^{r}}\lesssim 
 	\|u \|_{H^{r}}^2 ,\quad \|v^2\|_{H^{s-1}} \lesssim 
 	\| v\|_{H^{s}}^2.
\end{equation*}
In addition, from the Sobolev multiplication law (\cite[Corollary 3.16]{tao}), we obtain
 	\begin{equation}\label{estimate2}
 		\begin{split}
 			&\|uv\|_{H^{r}}\lesssim 
 			\|u \|_{H^{r}} \|v \|_{H^{s}},\quad 
 			\|v^2\|_{H^{r}}\lesssim 
 			\|v \|_{H^{s}}^2,
 			\\
 			&\|u^2\|_{H^{s-1}}\lesssim
 			\|u \|_{H^{r}}^2,\quad 
 	\| uv\|_{H^{s-1}}\lesssim
 			\|u \|_{H^{r}} \|v \|_{H^{s}}.
 		\end{split}
 	\end{equation}
Using \eqref{estimate2} in \eqref{estimate1}, \eqref{estimate1b}, a standard application of the Contraction Mapping Theorem determines some $T>0$ and the existence of a unique solution $(u,v)$ of \eqref{integralform} with  $(u,u_{t})\in C_{T}(X^{r}), (v.v_{t})\in C_{T}(X^{s})$. From \eqref{integralform} again, it is clear that actually $(u,u_{t})\in C^{1}_{T}(X^{r}), (v.v_{t})\in C_{T}^{1}(X^{s})$.
\end{proof}

 A second observation is concerned with the conserved quantities of \eqref{kgb}. A direct computation proves the following result.
 
 \begin{theorem}\label{hamiltonian}
 	Assume that 
 	\begin{equation}\label{ham1}
 		b_{uu}=-a_{uv}=:B,\quad b_{uv}=-a_{vv}=:C,
 	\end{equation}
 	and that \eqref{kgb} admits a unique solution  $(u,u_{t})\in C^{1}_{T}(X^{0}), (v.v_{t})\in C_{T}^{1}(X^{1})$. Let
 	\begin{equation} \label{energyc}   E(t)=\frac12\int_\rr\left(u_t^2+\al^2u^2+ (\partial_x^{-1}u_t)^2+v^2+v_x^2+v_t^2\right)\dd x
 		-\int_\rr\fff(u,v)\dd x,
 	\end{equation}
 	\begin{equation} \label{inva}
 		F(t)=\int_\rr \left(u\ned u_t+u_xu_t+v_xv_t\right)\dd x,
 	\end{equation}
 	where
 	\[ 
 	\fff(u,v)=-\frac{a_{uu}}{3}u^3+\frac{b_{vv}}{3}v^3+Bu^2v+Cuv^2.
 	\]
 	Then the functionals \eqref{energyc} and \eqref{inva} are conserved  for $t\in[0,T)$.
 \end{theorem}

\begin{remark}
 	We observe that, for an initial data $(u_0,u_1)\in X^{0}$, $(v_{0},v_{1})\in X^{1})$, the ivp \eqref{kgb} is equivalent to 
 	\begin{equation}\label{kgb-equiv}
 		\begin{cases}
 			u_{t}=w_{x},\\
 			(I-\partial_x^2)w_t=(\al^2u +f_1(u,v))_{x},\\
 			v_t=z,\\
 			z_t=v_{xx}-v+ f_2(u,v) \\
 			u(x,0)=u_0,\quad w(x,0)=\partial_x^{-1}u_1,\\
 			v(x,0)=v_0,\quad z(x,0)=v_1.
 		\end{cases}
 	\end{equation}
 	in the sense that $(u,v)$ with $(u,u_{t})\in C_{T}(X^{0}), (v.v_{t})\in C_{T}(X^{1})$  is a solution of \eqref{kgb} if, and only
 	if, $(u,w,v,z)\in C_{T}(L^{2}\times H^{1}\times X^{1})$ is a solution of \eqref{kgb-equiv}.  (Note here that $\partial_x^{-1}$is an invertible operator
 	from $H^1$ to $L^2\cap\dot{H}^{-1}$.) 
 	Then the  energy space associated to \eqref{kgb-equiv} is $\nx=L^2\times H^1\times X^{1}$, and, therefore,  the energy space associated with \eqref{kgb} is $X^{0}\times X^{1}$. The equivalence enables to establish an alternative proof of Theorem \ref{local} from \eqref{kgb-equiv}, cf. \cite{liu}. On the other hand, in terms of \eqref{kgb-equiv}, the invariants \eqref{energyc} and \eqref{inva} are written as
 	\[
 	E(t)=\frac12\int_\rr\left(\al^2u^2+ w^2+w_x^2+v^2+v_x^2+z^2\right)\dd x
 	-\int_\rr\fff(u,v)\dd x ,
 	\]
 	\[
 	F(t)=\int_\rr \left(uw+u_xw_x+v_xz\right)\dd x ,
 	\]
 	In addition, $E$ is the Hamiltonian function of the Hamiltonian structure of \eqref{kgb-equiv}
 	\[
 	\vec U=JE'(\vec U),\qquad \vec U=\begin{pmatrix}
 		u\\w\\v\\z
 	\end{pmatrix},\quad
 	J=\begin{pmatrix}
 		0&(I-\partial_x^2)^{-1}\partial_x&0&0\\
 		(I-\partial_x^2)^{-1}\partial_x&0&0&0\\
 		0&0&0&1\\
 		0&0&-1&0
 	\end{pmatrix},
 	\]
 	where $E'$ denotes the variational derivative.
 \end{remark}

The next theorem gives some conditions under which the solutions is global in the energy space. The proof requires the following auxiliary result.
	 \begin{lemma}\label{lem-gb}
 	Let $s>1$.	Assume that $y=y(t)$ is a continuous
 	function
 	satisfying 
 	\[
 	0\leq y(t)\leq C_1+C_2 (y(t))^s, 
 	\]
 	for all $t\geq0$ and for some $C_1,C_2>0$ such that $C_1<\frac{s-1}{s}(sC_2)^{\frac{1}{1-s}}$. Then, there are $a_{1}, a_{2}, A$ with $0<a_1<A<a_2<\infty$,  such that $0\leq y(t)\leq a_1$ if $y(0)<A$, and $a_2\leq y(t)$ if $y(0)>A$ for all $t\geq0$, where $A=(sC_2)^{\frac{1}{1-s}}$.
 \end{lemma}
	
 \begin{theorem}\label{global-1}
Consider the ivp \eqref{kgb} with $a_{uu}=0$. Under the conditions of of Theorem \ref{local} and assuming \eqref{ham1}, 
 	there exists $K_0>0$ such that if
 	\begin{equation}\label{global-c}
 		E(0)<\frac{1}{6}K_0^{-6},\qquad \|u_1 \|_\lt^2 + \al^2\| u_0\|_\lt^2+\|v_0\|_\lt^2+\|v_1\|_\lt^2+\|(v_0)_x\|_\lt^2<K_0^{-3},
 	\end{equation} 
 	then the maximal local time of
 	existence $T>0$ in Theorem \ref{local} can be extended to $+\infty$.
	\end{theorem}
\begin{proof}
 	Let $y(t)=\|\partial_x^{-1}u_t \|_\lt^2 + \al^2\| u\|_\lt^2+\|v\|_\lt^2+\|v_t\|_\lt^2+\|v_x\|_\lt^2$. 
By using the Sobolev embedding we get
 	\[
 	y(t)\leq 2E(0)+\frac23\left(
 	|b_{vv}|C_3^3  
 	+|B|C_\ast+
 	|C|C_4^2  
 	\right)y^\frac32(t)=
 	2E(0)+\frac23K_0 y^\frac32(t),
 	\]
 	where $K_{0}=|b_{vv}|C_3^3  
 	+|B|C_\ast+
 	|C|C_4^2 $ and  $C_3$, $C_4$ and $C_\ast$  are the best constants for which the estimates $\|f\|_{L^3}\leq C_3  \|f\|_{H^1}$, $\|f\|_{L^4}\leq C_4  \|f\|_{H^1}$, and $\|f\|_{L^\infty}\leq C_\ast \|f\|_{H^1}$ hold.
 	Hence, by using Lemma \ref{lem-gb}, if \eqref{global-c} holds,	
 	then $y(t)$ is bounded. 
 \end{proof}	
\begin{remark}
	It is known (see for instance \cite{nagy}) that the constants $C_3$ and $C_4$ in the proof of Theorem \ref{global-1} are represented by using the unique ground states
	 \[
	 \phi(x)=\left(\frac r2\right)^{\frac1{r-2}}
	 {\rm sech}^{\frac2{r-2}}\left(\frac{r-2}{2}x\right)
	  \] of
	\[
-\phi''+\phi=\phi^{r-1},\qquad r=3,4.
	\]
	Indeed, there holds that
	\[
	C_r^r=\frac{2r}{2+r}\left(\frac{2+r}{r-2}\right)^{\frac{r-2}{4}} \|\phi\|_{L^2}^{2-r}
	=\begin{cases}
  5^{-\frac34}\sqrt{6} &r=3\\
 \frac1{\sqrt{2} }	&r=4
\end{cases}.
		\]
		Hence, if $B=0$, then in  \eqref{global-c} we have 
		\[
  K_0=5^{-\frac34}\sqrt{6}	|b_{vv}|
	 +
	  2^{-\frac14}  \sqrt{|C|} .
		\]
\end{remark}
	 
On the other hand, using the approach in \cite{liu} and the techniques given by \cite{levine},
the following blow-up result holds.
	\begin{theorem}\label{theo-blowup}
	We assume the conditions \eqref{ham1}.	Let $u_0,u_1,v_0,v_1$ be as in Theorem \ref{local}  and, in addition, $\partial_x^{-1}u_{0}\in L^2$. 	Then the local solution $   (u,v)\in C^1([0,T);H^s\times H^{s+1} )$ blows up in finite time if one the following cases holds:
	\begin{enumerate}[(i)]
		\item $E(0) <0$,
		\item   $E(0)\geq0$ and
		\begin{equation}\label{st1}
		(2E(0))^{\frac12}<\frac{\scal{\xi^{-1}\hat{u}_0, \xi^{-1}\hat{u}_1}+\scal{u_0,u_1}+\scal{v_0,v_1}}{\sqrt{\|\partial_x^{-1}u_0\|_\lt^2 + \| u_0\|_\lt^2+\|v_0\|_\lt^2}}.
		\end{equation}
	\end{enumerate}
	\end{theorem}	
		\begin{proof}
			Define $I(t)=\|\partial_x^{-1}u \|_\lt^2 + \| u\|_\lt^2+\|v\|_\lt^2+\beta(t+t_0)^2$, where $\beta,t_0\geq0$ will be determined later.
			Then we have
			\[
			\frac12I'(t)=\scal{\ned u,\ned u_t}+ \scal{u, u_t}+\scal{v,v_t}+\beta(t+t_0),
			\]
			and
			\[
			\begin{split}
				\frac12I''(t)
				&	=
				\norm{\ned u_t}_\lt^2+
				\norm{ v_t}_\lt^2+	 	\norm{  u_t}_\lt^2
				-\norm{ v_x}_\lt^2	-\norm{ v }_\lt^2-\al^2 \norm{ u }_\lt^2+\beta+\scal{v,f_2(u,v)}
				-\scal{u,f_1(u,v)}	\\
				&=	
				\frac52\left(\norm{\ned u_t}_\lt^2+
				\norm{ v_t}_\lt^2+	 	\norm{  u_t}_\lt^2\right)
				+\frac12\left( \norm{ v_x}_\lt^2+\norm{ v }_\lt^2+\al^2 \norm{ u }_\lt^2\right)	+\beta-3E(0)	.		
			\end{split}
			\]
			After some direct calculations, we observe that
			\begin{equation}\label{st1b}
			I''(t)I(t)-\frac54(I'(t))^2\geq -3(2E(0)+\beta)I(t),
			\end{equation}	
			where
\begin{equation}\label{bl-1b}
				I_1''(t)=-\frac14I^{-\frac94}(t)\left(I(t)I''(t)-\frac54(I'(t))^2\right),\quad t\geq 0.
			\end{equation} 
			
			If $E(0)<0$, we can choose $\beta<-2E(u_0,v_0)$, so that, from \eqref{st1b}
\begin{equation}\label{st1c}
			I''I-\frac54(I')^2>0.
			\end{equation}
We can assume that $I'(0)>0$ by choosing $t_0>0$ sufficiently large. Define $I_1(t)=(I(t))^{-1/4}$. Then, from \eqref{bl-1b}, \eqref{st1c}, we have $I_1''(t)<0, t>0$.
Therefore, $I_{1}(t)\leq I_{1}(0)+I_{1}'(0)t$, \cite{liu}. Since $I_{1}(0)>0$ and $I_{1}'(0)<0$, then there is som $t^{*}$ with $0<t^{*}<-I_{1}(0)/I_{1}'(0)$ such that $I_{1}(t^{*})=0$ and $I(t)$ blows up at some finite time in the interval $(0,4I(0)/I'(0))$.

If $E(0)=0$, we take $\beta=0$, so that
			\[
			I''I-\frac54(I')^2\geq0.
			\]
Then, from \eqref{bl-1b}, $I_1''(t)\leq 0$ and from \eqref{st1}
$$I_{1}'(0)=-\frac{I'(0)}{4I(0)^{5/4}}<0.$$ Thus, $I(t)$ blows up at finite time with the same argument as in the previous case.

			If $E(0)>0$, we also take $\beta=0$.  Then, $I_1(0)>0$ and, from \eqref{st1}, $I_1'(0)<0$. Furthermore
from \eqref{bl-1b} and \eqref{st1b}

			\begin{equation}\label{bl-1}
				I_1''(t)
				\leq \frac32E(0)I^{-\frac54}(t),
			\end{equation} 
			for $t>0$. Let $t_\ast=\sup\{t,\;I'_1(\tau)<0 \,\text{for}\,\tau\in[0,t)\}
			$. Note that the continuity of $I_1$ implies  $t_\ast>0$. For $t\in [0,t_{*})$, We multiply  \eqref{bl-1} by $2I_1'(t)$ to get
			\[
			\frac\dd{\dd t}((I_1'(t))^2)
			\geq
			\frac12E(0)\left(\frac{\dd}{\dd t}I^{-\frac32}(t)\right).
			\]
			Integrating over $[0,t)$ leads to
			\begin{equation}\label{st2}
			(I'_1(t))^2\geq
			\frac12E(0)I^{-\frac32}(t)+(I_1'(0))^2-\frac12E(0)I^{-\frac32}(0).
			\end{equation}
From \eqref{st1}, we have 
			\[
			(I'_1(0))^2>\frac12
			E(0)I^{-\frac32}(0).
			\]
Hence, from \eqref{st2}, it holds that
\begin{equation}\label{st3}
|I_1'(t)|\geq \sqrt{(I_1'(0))^2-\frac12E(0)
				I^{-\frac32}(0)}\Rightarrow I_1'(t)\leq -\sqrt{(I_1'(0))^2-\frac12E(0)
				I^{-\frac32}(0)}<0,\quad t\in [0,t_{*}).
\end{equation}
From the continuity of $I_{1}'(t)$ and the definition of $t_{*}$, it follows that $t_{*}=\infty$ and \eqref{st3} holds for all $t\geq 0$. Integrating over $(0,t), t>0$, we have
\begin{equation*}
I_1(t)\leq I_1(0)-\sqrt{(I_1'(0))^2-\frac12E(0)
				I^{-\frac32}(0)}\;t.
\end{equation*}
Thus, $I_1(t^\ast)=0$ for some $t^\ast\in(0,\ell]$, where $\ell=I_1(0)/\sqrt{(I_1'(0))^2-\frac12E(0)
				I^{-\frac32}(0)}$, and
\begin{equation*}
			\|\partial_x^{-1}u \|_\lt^2 + \| u\|_\lt^2+\|v\|_\lt^2\to+\infty,
			\end{equation*}
			as $t\to t^\ast$.

		\end{proof}
\section{Solitary-wave solutions of Boussinesq Klein-Gordon system}\label{solit-sec}
This section is devoted to the existence of solitary wave solutions of \eqref{kgb1}. These are smooth traveling-wave solutions $u(x,t)=u(x-c_{s}t), v(x,t)=v(x-c_{s}t)$, with $c_{s}\neq 0$ and such that the derivatives $u^{j)}(X), v^{j)}(X)\rightarrow 0$ as $|X|\rightarrow\infty, X=x-c_{s}t$ for $j=1(1)3$. Then the profiles $u,v$ must satisfy the coupled system
\begin{eqnarray}
\left((c_{s}^{2}-\alpha^{2})-c_{s}^{2}\partial_{x}^{2}\right)u&=&f_{1}(u,v),\nonumber\\
\left(1-(1-c_{s}^2)\partial_{x}^{2}\right)v&=&f_{2}(u,v).\label{gm2}
\end{eqnarray}
\subsection{Existence via linearization}
\label{nft}
One of the approaches to study the existence of solutions of (\ref{gm2}) is based on Normal Form theory, \cite{IK,Champ,ChampS,ChampT}. We write (\ref{gm2}) as a first-order differential system for $U=(U_{1},U_{2},U_{3},U_{4})^{T}:=(u,u',v,v')^{T}$, as
\begin{eqnarray}
U^{\prime}&=&V(U,c_{s}):=L(c_{s})U+R(U,c_{s}),\label{gm2b}\\
L(c_{s})&=&\begin{pmatrix}0&1&0&0\\\mu&0&0&0\\
0&0&0&1\\0&0&\frac{1}{1-c_{s}^{2}}&0\end{pmatrix},\quad
R(U,c_{s})=\begin{pmatrix}0\\\frac{-1}{c_{s}^{2}}f_{1}(U)\\
0\\\frac{-1}{1-c_{s}^{2}}f_{2}(U)\end{pmatrix},\nonumber\\
f_{1}(U)&=&a_{uu}U_{1}^{2}+2a_{uv}U_{1}U_{3}+a_{vv}U_{3}^{2},\nonumber\\
f_{2}(U)&=&b_{uu}U_{1}^{2}+2b_{uv}U_{1}U_{3}+b_{vv}U_{3}^{2},\label{gm2c}
\end{eqnarray}
where $\mu=1-\alpha^{2}/c_{s}^{2}$. Note that the system (\ref{gm2b}), (\ref{gm2c}) admits $U=0$ as solution and the vector field $V$ is reversible, in the sense that for all $U, c_{s}$,
\begin{equation}\label{rev}
SV(U,c_{s})=-V(SU,c_{s}),
\end{equation} 
where $S={\rm diag}(1,-1,1,-1)$. These properties enable to study the existence of solutions of (\ref{gm2b}), (\ref{gm2c}), for small values of $\mu$, by using the Normal Form theory, analyzing first the linearization at $U=0$. The characteristic equation is
\begin{eqnarray}
\lambda^{4}-B\lambda^{2}+A=0,\label{gm3}
\end{eqnarray}
where
\begin{eqnarray}
A=\frac{c_{s}^{2}-\alpha^{2}}{c_{s}^{2}(1-c_{s}^{2})}=\frac{\mu}{1-c_{s}^{2}},\;
B=\mu +\frac{1}{1-c_{s}^{2}}.\label{gm3b}
\end{eqnarray}
The spectrum of the linearization can be studied by using \cite{Champ}. The distribution of the roots in the $(B,A)$-plane is sketched in Figure \ref{fig_A1}, which reproduces the bifurcation diagram, along with the location and the type of the four eigenvalues, shown in Figure 1 of \cite{Champ}. 
\begin{figure}[htbp]
\centering
{\includegraphics[width=1\textwidth]{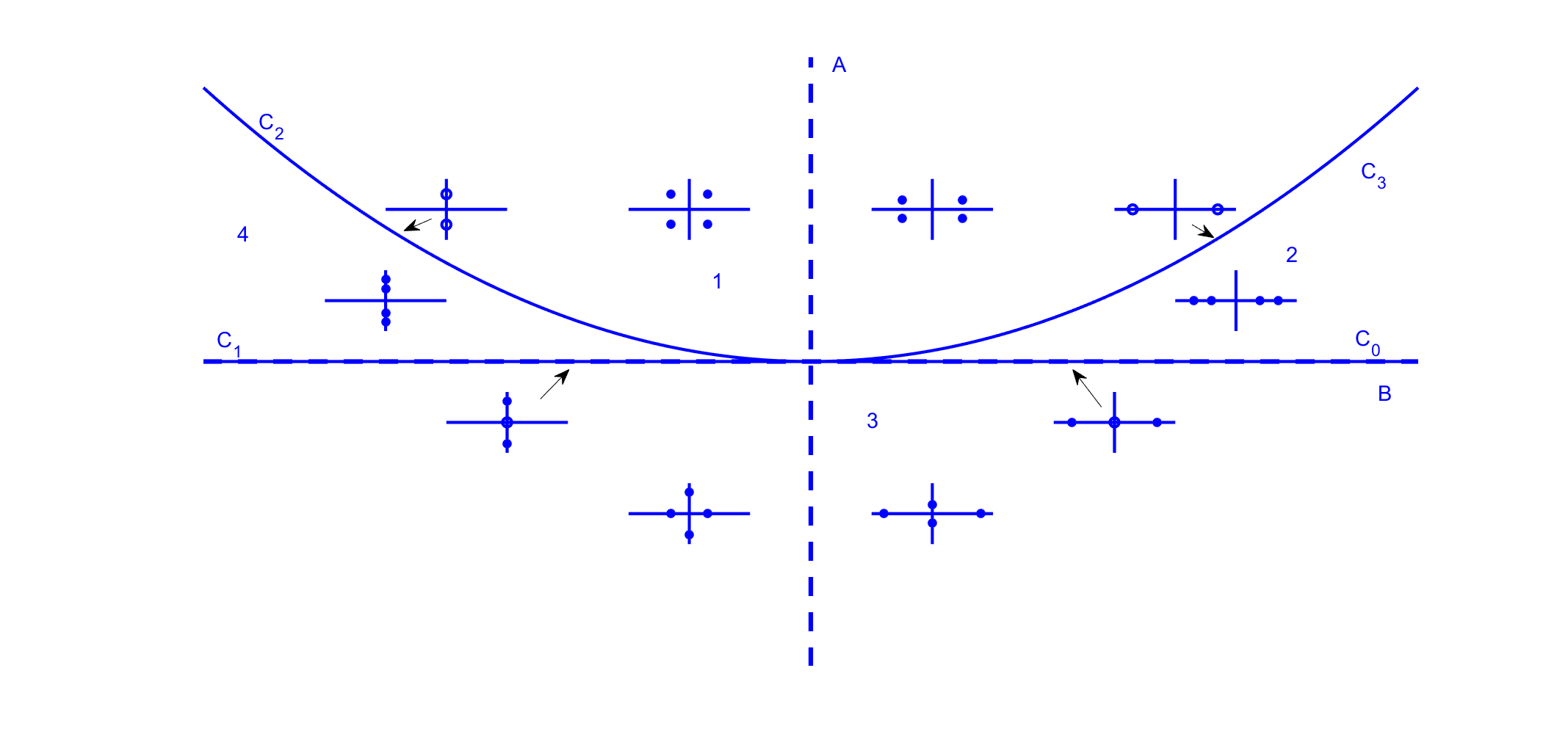}}
\caption{Linearization at the origin of (\ref{gm2b}) (cf. Figure 1 of \cite{Champ}): Regions $1$ to $4$ in the $(B,A)$-plane, delimited by the bifurcation curves $\mathbb{C}_{0}$ to $\mathbb{C}_{3}$ given by (\ref{bifurcurv}), and schematic representation of the position in the complex plane of the roots of (\ref{gm3}) for each curve and region. (Dot: simple root, larger dot: double root.)}
\label{fig_A1}
\end{figure}
There we can distinguish four regions delimited by the bifurcation curves
\begin{eqnarray}
\mathbb{C}_{0}&=&\{(B,A) / A=0, B>0\},\nonumber\\
\mathbb{C}_{1}&=&\{(B,A) / A=0, B<0\},\nonumber\\
\mathbb{C}_{2}&=&\{(B,A) / A>0, B=-2\sqrt{A}\},\nonumber\\
\mathbb{C}_{3}&=&\{(B,A) / A>0, B=2\sqrt{A}\}.\label{bifurcurv}
\end{eqnarray}
The Center Manifold Theorem and the theory of reversible bifurcations can be applied to study the existence of homoclinic orbits in each bifurcation. The reduced Normal Form systems reveal the existence of two types of trajectories: homoclinic to zero and homoclinic to periodic orbits. The associated solutions correspond to classical solitary waves (CSW's) and generalized solitary waves (GSW's), respectively. In addition, periodic and quasi-periodic solutions can be identified, \cite{IK}. 

Before applying the approach in \cite{Champ} near the bifurcation curves $\mathbb{C}_{0}$ to $\mathbb{C}_{3}$ to the case of (\ref{gm2b}), using $\mu$ as bifurcation parameter, we first make a description of regions and curves presented in Figure \ref{fig_A1} and according to the values of $A$ and $B$ given by (\ref{gm3b}).  
 Note first that $\mathbb{C}_{0}$ is characterized by
\begin{eqnarray*}
\mu=0 \;(c_{s}=\pm\alpha), \; |c_{s}|<1 \;(B>0)
\end{eqnarray*}
while the conditions for the curve $\mathbb{C}_{1}$ are
\begin{eqnarray*}
\mu=0 \;(c_{s}=\pm\alpha), \; |c_{s}|>1\; (B<0)
\end{eqnarray*}
Observe now that $$B^{2}-4A=(\mu-\frac{1}{1-c_{s}^{2}})^{2},$$ so $B^{2}-4A=0$ if and only if
\begin{eqnarray}
\frac{1}{1-c_{s}^{2}}=\mu=\frac{c_{s}^{2}-\alpha^{2}}{c_{s}^{2}}.\label{gm3c}
\end{eqnarray}
Let us study condition (\ref{gm3c}). We have two possibilities:
\begin{itemize}
\item[(P1)] If $\mu>0$ then $1-c_{s}^{2}>0$. This means that $\alpha^{2}<c_{s}^{2}<1$.
\item[(P2)] If $\mu<0$ then $1-c_{s}^{2}<0$ and, therefore, $1<c_{s}^{2}<\alpha^{2}$.
\end{itemize}
On the other hand, (\ref{gm3c}) leads to the quadratic equation for $c_{s}^{2}$
$$c_{s}^{4}-\alpha^{2}c_{s}^{2}+\alpha^{2}=0,$$ yielding
$$c_{s}^{2}=\widetilde{c}_{\pm}(\alpha):=\frac{\alpha^{2}\pm\sqrt{\alpha^{4}-4\alpha^{2}}}{2},$$ which requires $\alpha^{2}\geq 4$. However, note that $\widetilde{c}_{\pm}(\alpha)$ satisfy the properties
$$\widetilde{c}_{-}(\alpha)<\widetilde{c}_{+}(\alpha)<\alpha^{2},\quad
\widetilde{c}_{-}(\alpha)<\widetilde{c}_{+}(\alpha)<1.$$ Therefore, it is not possible to have any of the two possibilities (P1) or (P2). Thus $B^{2}-4A>0$ and for the case at hand the curves $\mathbb{C}_{2}, \mathbb{C}_{3}$ are not present.
Furthermore, if we additionally assume $A> 0$ then $|B|>2\sqrt{A}$. That is, $B>2\sqrt{A}$ or $B<-2\sqrt{A}$. Consequently, region 1 is empty here.

Using (\ref{gm3b}) we can characterize regions 2 and 4. In the first case, where $A, B>0$ and $B>2\sqrt{A}$, observe that if $\mu<0$, then the form of $A$ in (\ref{gm3b}) implies $1-c_{s}^{2}<0$ and the form of $B$ gives $B<0$, which is not possible. Therefore $\mu>0$ and from (\ref{gm3b}) we have $1-c_{s}^{2}>0$; thus region 2 is characterized by
\begin{eqnarray}
\alpha^{2}<c_{s}^{2}<1.\label{gm3d}
\end{eqnarray}
Similarly, it is not hard to see that region 4 (for which $A>0, B<0, B<-2\sqrt{A}$) we must have $\mu<0$ and then this region is described as
$$1<c_{s}^{2}<\alpha^{2}.$$
Finally, region 3 can be divided into two subregions:
\begin{itemize}
\item Region 3R (right): $A<0, B>0$.
\item Region 3L (left): $A<0, B<0$.
\end{itemize}
Consider first region 3R and assume $\mu>0$. Then, using \eqref{gm3b}, the conditions  $A<0, B>0$ hold when
\begin{equation}\label{*3}
c_{s}^{2}>\alpha^{2},\quad c_{s}^{2}>1,\quad p_{\alpha}(c_{s}^{2})>0,
\end{equation}
where 
\begin{eqnarray*}
p_{\alpha}(z)&=&z^{2}-(2+\alpha^{2})z+\alpha^{2}=(z-z_{+}(\alpha))(z-z_{-}(\alpha)),\nonumber\\
z_{\pm}(\alpha)&=&\frac{1}{2}\left((2+\alpha^{2})\pm\sqrt{\alpha^{4}+4}\right),\label{gm3e}
\end{eqnarray*}
We note that
\begin{equation}\label{*3b}
z_{-}(\alpha)<\alpha^{2}<z_{+}(\alpha),\quad 
z_{-}(\alpha)<1<z_{+}(\alpha).
\end{equation}
Due to \eqref{*3}, this necessarily implies that $c_{s}^{2}>z_{+}(\alpha)$. Similarly, when $\mu<0$, conditions  $A<0, B>0$ hold when
\begin{equation*}
c_{s}^{2}<\alpha^{2},\quad c_{s}^{2}<1,\quad p_{\alpha}(c_{s}^{2})<0,
\end{equation*}
which, from \eqref{*3b}, implies
\begin{eqnarray*}
z_{-}(\alpha)<c_{s}^{2}<1<z_{+}(\alpha),\quad z_{-}(\alpha)<c_{s}^{2}<\alpha^{2}<z_{+}(\alpha)
\end{eqnarray*}
Similar arguments can be used to describe region 3L. All this is summarized in Table \ref{KGB_tav1}.

\begin{table}[ht]
\begin{center}
\begin{tabular}{ |c|c|c|c|  }
 \hline
 \multicolumn{2}{|c|}{Region 3R ($A<0, B>0$)}& \multicolumn{2}{|c|}{Region 3L ($A<0, B<0$)} \\
 \hline
 $\mu>0$& $\mu<0$&$\mu>0$& $\mu<0$\\
 \hline
 $c_{s}^{2}>z_{+}(\alpha)$ & $z_{-}(\alpha)<c_{s}^{2}<\min\{1,\alpha^{2}\}$   &$\min\{1,\alpha^{2}\}<c_{s}^{2}<z_{+}(\alpha)$&  $c_{s}^{2}<z_{-}(\alpha)$\\
 \hline
\end{tabular}
\end{center}
\caption{Description of region 3 in Figure \ref{fig_A1} for the case of \eqref{gm2b}.}
\label{KGB_tav1}
\end{table}


We now study the information provided by the Normal Form Theory (NFT) close to each curve $\mathbb{C}_{j}, 0\leq j\leq 3$.
In the case of $\mathbb{C}_{0}$, the linearization matrix $L(\pm\alpha)$ has two simple eigenvalues equal to
\begin{eqnarray}
\lambda_{\pm}=\pm\frac{1}{\sqrt{1-\alpha^{2}}},\label{NFT3}
\end{eqnarray}
and the zero eigenvalue with geometric multiplicity one and algebraic multiplicity two. As in \cite{IK,Champ}, the main role in describing the dynamics close to $\mathbb{C}_{0}$ by NFT is played by this two-dimensional center manifold. When ${\mu}$, $A$ and $B$ are positive, and near $\mathbb{C}_{0}$ the linear dynamics is given by the spectrum of $L({\mu})$ which consists of four real eigenvalues (region 2 in Figure \ref{fig_A1}). In this case, the normal form system has a unique solution, homoclinic to zero at infinity, symmetric and unique up to spatial translations, (\cite{IK}, Proposition 3.1), that corresponds to a CSW solution of \eqref{gm2}.
The form of the waves is illustrated in Figure~\ref{Fig:gm1}.
\begin{figure}
  \centering
  \subfigure[]
  {\includegraphics[width=0.49\textwidth]{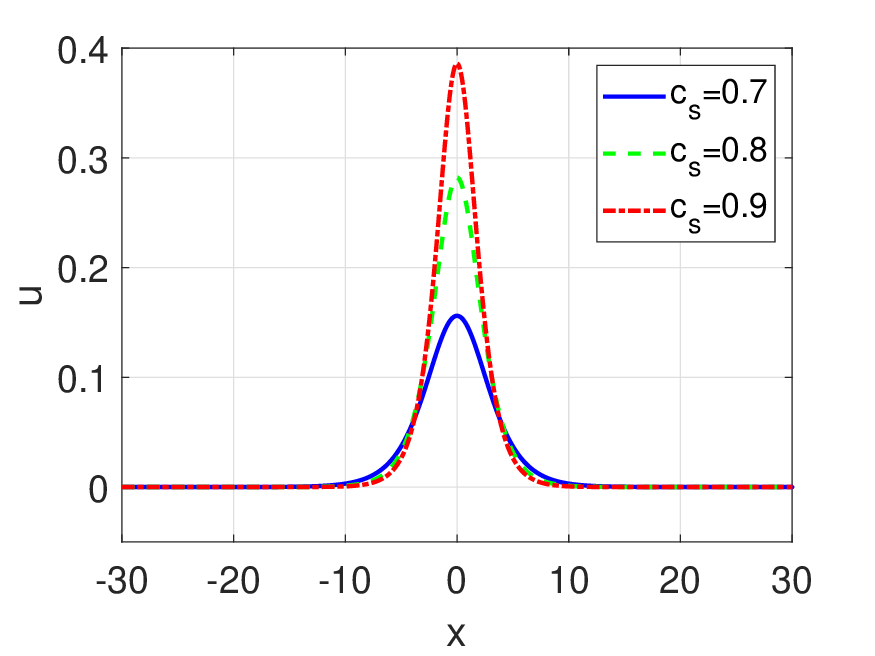}}
  \subfigure[]
  {\includegraphics[width=0.49\textwidth]{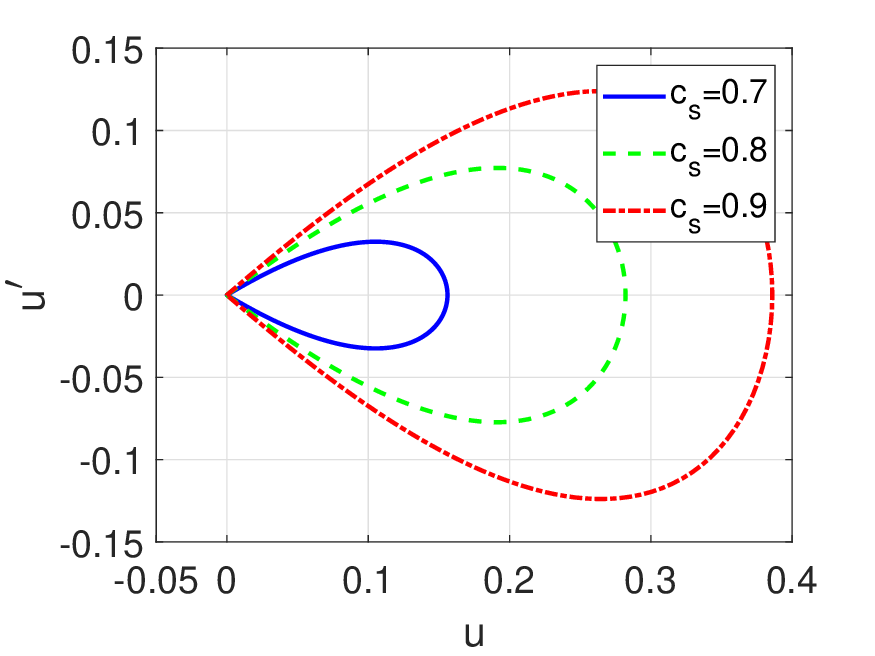}}
  \subfigure[]
  {\includegraphics[width=0.49\textwidth]{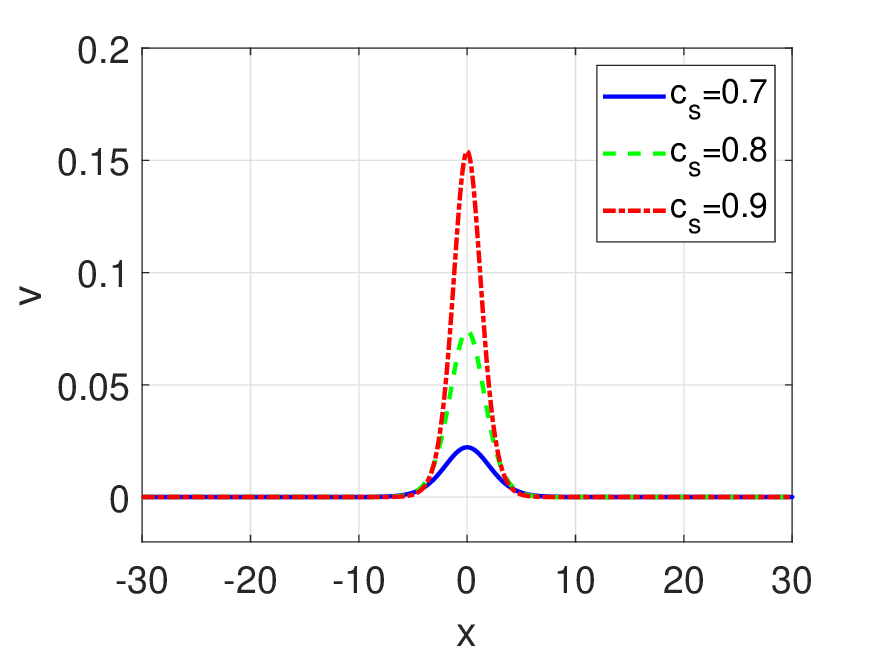}}
  \subfigure[]
  {\includegraphics[width=0.49\textwidth]{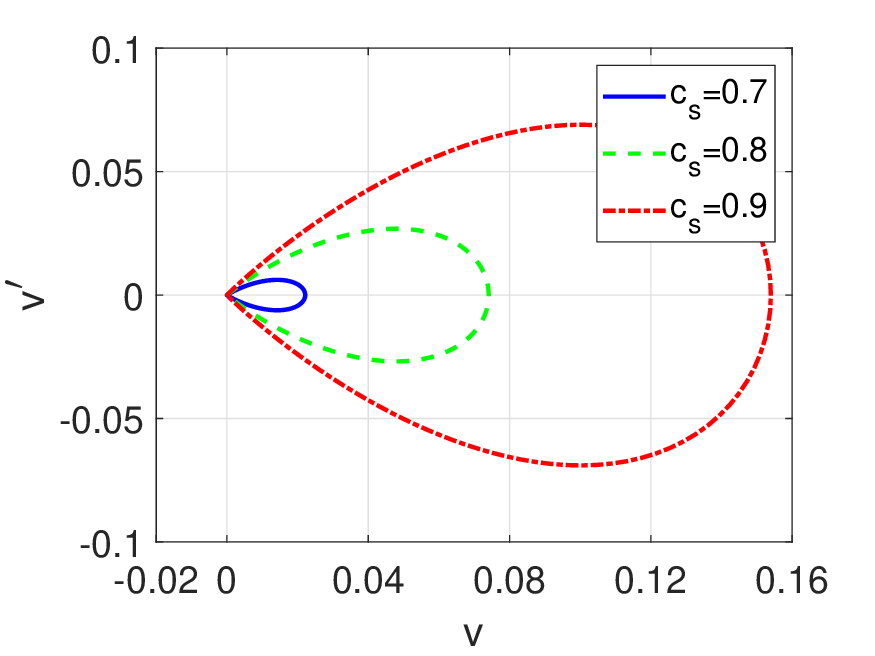}}
  \caption{Approximate classical solitary wave profile solutions of (\ref{kgb}) with $f_{1}(u,v)=(u+v)^{2}$, $f_{2}(u,v)=u^{2}+v^{2}$, $\alpha=0.6$, and several speeds. (a) $u$ profiles; (b) phase portraits of (a); (c) $v$ profiles; (d) phase portraits of (c).}
  \label{Fig:gm1}
\end{figure}

Let $\{w_{1},w_{2},w_{3},w_{4}\}$ be a basis of generalized eigenvectors of $L(\pm\alpha)$, with $w_{3}$, $w_{4}$ eigenvectors of $\lambda_{+}$ and $\lambda_{-}$, resp., $w_{1}$ eigenvector associated to the zero eigenvalue and $w_{2}$ such that $L(\pm\alpha)w_{2}=w_{1}\,$. Explicitly, we take
\begin{align*}
  w_{1}\ =(1,0,0,0)^{T}, \qquad & w_{2}\ =(0,1,0,0)^{T},\\
  w_{3}\ =(0,0,1,\lambda_{+})^{T}, \qquad & w_{4}\ =(0,0,1,\lambda_{-})^{T}.
\end{align*}
Note that $w_{1}$, $w_{2}$ additionally satisfy $Sw_{1}=w_{1}$, $Sw_{2}\ =-w_{2}$, where $S$ is given by \eqref{rev}. If $P$ is the matrix with columns given by the $w_{j}$'s, then we consider the new variables $V=(v_{1},v_{2},v_{3},v_{4})$ such that $U=PV$. The system (\ref{gm2b}) in the new variables takes the form
\begin{eqnarray*}
  v_{1}'&=&v_{2}, \\
  v_{2}'&=&\mu v_{1}-\frac{1}{c_{s}^{2}}\left(a_{uu}v_{1}^{2}+2a_{uv}v_{1}(v_{3}+v_{4})+a_{vv}(v_{3}+v_{4})^{2}\right)\\
  v_{3}'&=&\alpha_{12} v_{3}+\beta_{1} v_{4}- \frac{\beta}{1-c_{s}^{2}}\left(b_{uu}v_{1}^{2}+2b_{uv}v_{1}(v_{3}+v_{4})+b_{vv}(v_{3}+v_{4})^{2}\right)\\
  v_{4}'&=&-\beta_{1} v_{3}-\alpha_{12} v_{4}+ \frac{\beta}{1-c_{s}^{2}}\left(b_{uu}v_{1}^{2}+2b_{uv}v_{1}(v_{3}+v_{4})+b_{vv}(v_{3}+v_{4})^{2}\right),
\end{eqnarray*}
where, since $\frac{1}{1-c_{s}^{2}}=\frac{1}{1-\alpha^{2}}+O(\mu)$ as $\mu\rightarrow 0$, then
\begin{eqnarray*}
\alpha_{12}&=&\frac{-\frac{1}{1-\alpha^{2}}-\frac{1}{1-c_{s}^{2}}}{-\frac{2}{\sqrt{1-\alpha^{2}}}}=\lambda_{+}+O(\mu),\\
\beta_{1}&=&\frac{\frac{1}{1-\alpha^{2}}-\frac{1}{1-c_{s}^{2}}}{-\frac{2}{\sqrt{1-\alpha^{2}}}}=+O(\mu),
\end{eqnarray*}
and $\beta=\frac{1}{\lambda_{+}-\lambda_{-}}$. Then $-\alpha_{12}=\lambda_{-}+O(\mu)$ and if $||V||\rightarrow 0$ ($||\cdot||$ denotes the Euclidean norm in $\mathbb{C}^{4}$) then
\begin{eqnarray}
  v_{1}'&=&v_{2}, \label{gm4a}\\
  v_{2}'&=&\mu v_{1}-\frac{a_{uu}}{c_{s}^{2}}v_{1}^{2}+O(||V||_{2}^{2}),\label{gm4b}\\
  v_{3}'&=&\lambda_{+}v_{3}+O(\mu||V||_{2}+||V||_{2}^{2}),\label{gm4c}\\
  v_{4}'&=&\lambda_{-}v_{4}+O(\mu||V||_{2}+||V||_{2}^{2}),\label{gm4d}
\end{eqnarray}
where we assume $a_{uu}>0$. Then the center-manifold reduction theorem, \cite{IA}, ensures the existence of bounded solutions of the (\ref{gm4a})-(\ref{gm4d}) on a locally invariant, center manifold determining a dependence $(v_{3},v_{4})=h(\mu,v_{1},v_{2})$ for some smooth $h(\mu,v_{1},v_{2})=O(\mu||(v_{1},v_{2})|| +||(v_{\,1},\,v_{\,2})||^{2})$ as $\mu$, $||(v_{1},v_{2})|| \ \rightarrow 0$, see \cite[Theorem 3.2]{IK}. Furthermore, every solution $v_{\,1}\,$, $v_{\,2}$ of the reduced system (\ref{gm4a})-(\ref{gm4b}) with $(v_{\,3},\,v_{\,4})\ =\ h\,(\,\mu,\,v_{\,1},\,v_{\,2}\,)$ induces a solution of (\ref{gm4a})-(\ref{gm4d}). The normal form system can be written as
\begin{equation*}\label{eq:snf}
  v_{1}'=v_{2},\quad v_{2}'={\rm sign}(\mu)\,v_{1}-\frac{3}{2}v_{1}^{2} +O(\mu),
\end{equation*}
which admits, for $\mu>0$, a solution of the form $v_{1}(x)={\rm sech}^{2}(x/2)+O(\mu)$, $v_{2}=v_{1}'$. For the persistence of this homoclinic orbit from the perturbation connecting to the original system (\ref{gm2b}), (\ref{gm2c}), see \cite{IK,Champ}.

In the case of $\mathbb{C}_{1}$, the spectrum of $L(\pm\alpha)$ consists of zero (with algebraic multiplicity two) and the two simple imaginary eigenvalues given by (\ref{NFT3}) (recall that $c_{s}^{2}>1$). The arguments used in \cite{IK}, Proposition 3.2, apply
here and NFT reduces (\ref{gm2b}), on the center manifold, for ${\mu}> 0$ small enough, to a normal form system which admits homoclinic solutions to periodic orbits, that is GSW solutions. Information about the structure of the periodic orbits can also be obtained, cf. \cite{Lombardi2000} and references therein. For our particular case, the basis $\{w_{1},w_{2},w_{3},w_{4}\}$, in $\mathbb{C}^{4}$, with $w_{1}$, $w_{2}$ as above, contains the eigenvectors
\begin{eqnarray*}
  w_{3}=(0,0,1,\frac{i}{\sqrt{\alpha^{2}-1}})^{T},\; 
  w_{4}=(0,0,1,\frac{-i}{\sqrt{\alpha^{2}-1}})^{T},
\end{eqnarray*} 
associated to $\pm\frac{i}{\sqrt{\alpha^{2}-1}}$, respectively. Following \cite{Lombardi2000} (see also \cite{IK}), let $\{w_{1}^{*},w_{2}^{*},w_{3}^{*},w_{4}^{*}\}$ be the corresponding dual basis (with, in particular, $w_{2}^{*}=w_{2}$). If $D_{\mu,U}^{2}V(U,\mu)$ denotes the derivative, with respect to $\mu$, of the {Jacobian} matrix of $V$ in \eqref{gm2b} and $D_{U,U}V(U,\mu)^{2}$ the {Hessian} operator of $V$, then
\begin{eqnarray*}
  c_{10}&:=&\langle\, w_{2}^{*},D_{\mu,U}^{2}V(0,\,0)w_{1}\rangle=1, \\
  c_{20}&:=&\frac{1}{2}\langle w_{2}^{*},D_{U,U}^{2}V(0,0)[w_{1},w_{1}]\rangle=-\frac{1}{c_{s}^{2}}a_{uu} \neq\ 0.
\end{eqnarray*}
Therefore, from \cite[Theorem~7.1.1]{Lombardi2000} (see also \cite{DDM2019}), \eqref{gm2b} admits, for $\mu$ small enough and near $U=0$, two orbits homoclinic to a one-parameter family of periodic orbits of arbitrarily small amplitude (see \cite{IK} for the application of NFT to the reduced system in this case).
The form of the waves is illustrated in Figure~\ref{Fig:gm2}. 
\begin{figure}
  \centering
  \subfigure[]
  {\includegraphics[width=0.49\textwidth]{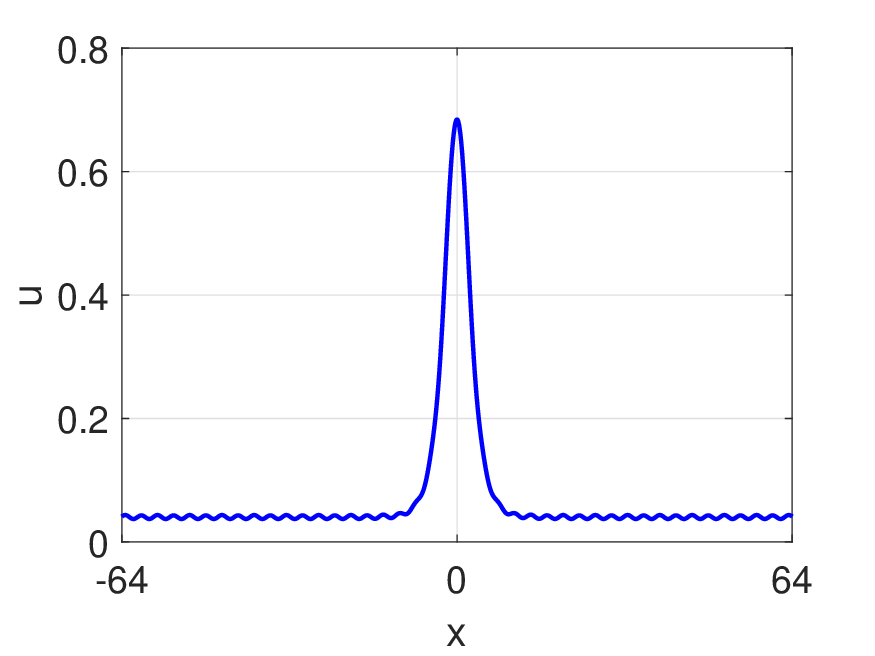}}
  \subfigure[]
  {\includegraphics[width=0.49\textwidth]{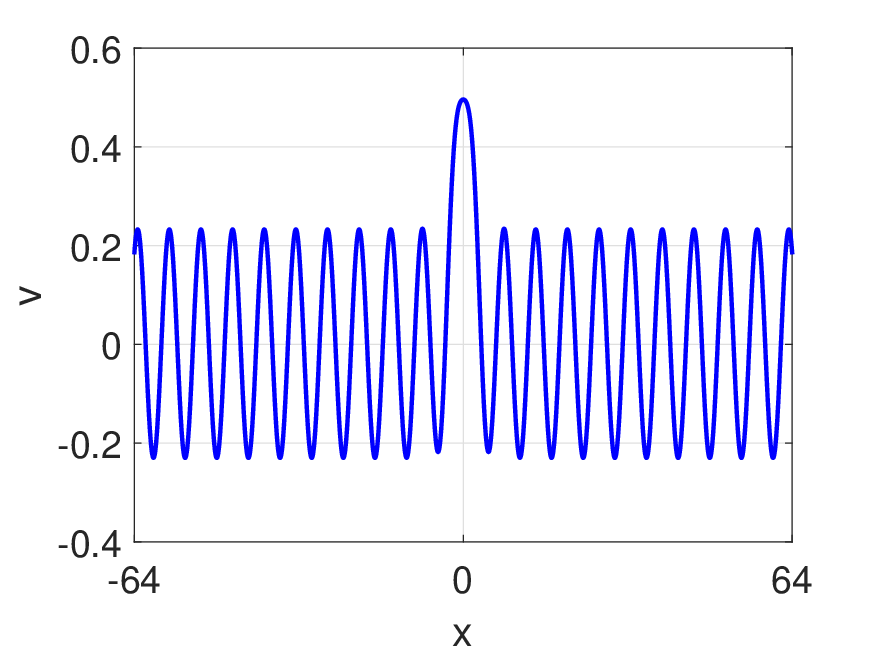}}
  \subfigure[]
  {\includegraphics[width=0.49\textwidth]{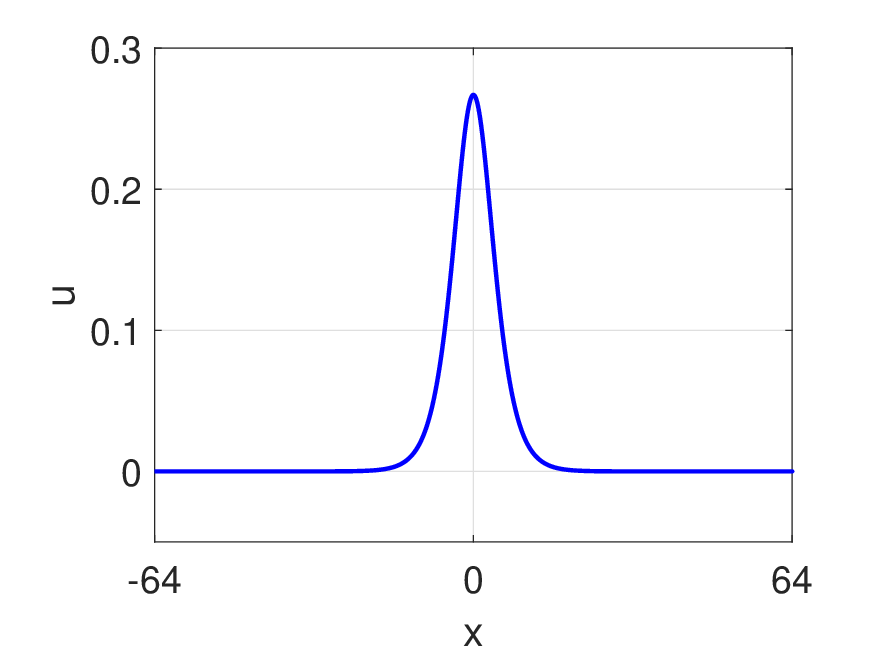}}
  \subfigure[]
  {\includegraphics[width=0.49\textwidth]{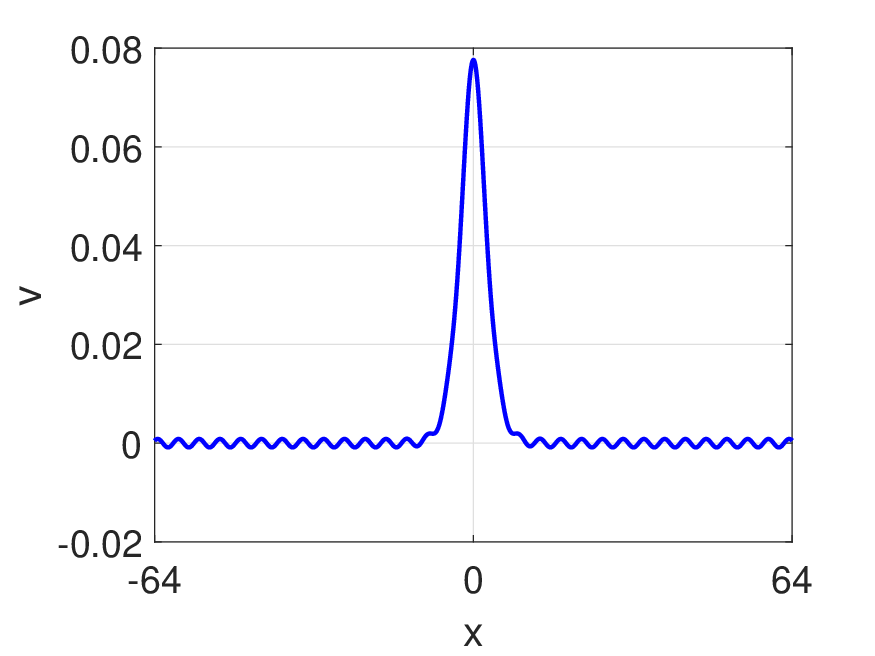}}
  \caption{ Approximate generalized solitary wave profile solutions of (\ref{kgb}) with $f_{1}(u,v)=u^{2}+v^{2}$, $f_{2}(u,v)=u^{2}$. (a), (b) $\alpha=1.12, c_{s}=1.4$; (c), (d) $\alpha=1.12, c_{s}=1.2$. (Region 3 of Figure \ref{fig_A1}.)}
  \label{Fig:gm2}
\end{figure}

\begin{remark}
If ${\mu}$ is negative (with $|\mu|$ small), by similar arguments to those of \cite{IK}, NFT establishes the existence of a family of periodic solutions of the reduced system (region 3, close to $\mathbb{C}_{0}$ and region 4, close to $\mathbb{C}_{1}$), unique up to spatial translations, see Figure \ref{Fig:gm3}.
\end{remark} 
\begin{figure}
  \centering
  \subfigure[]
  {\includegraphics[width=0.49\textwidth]{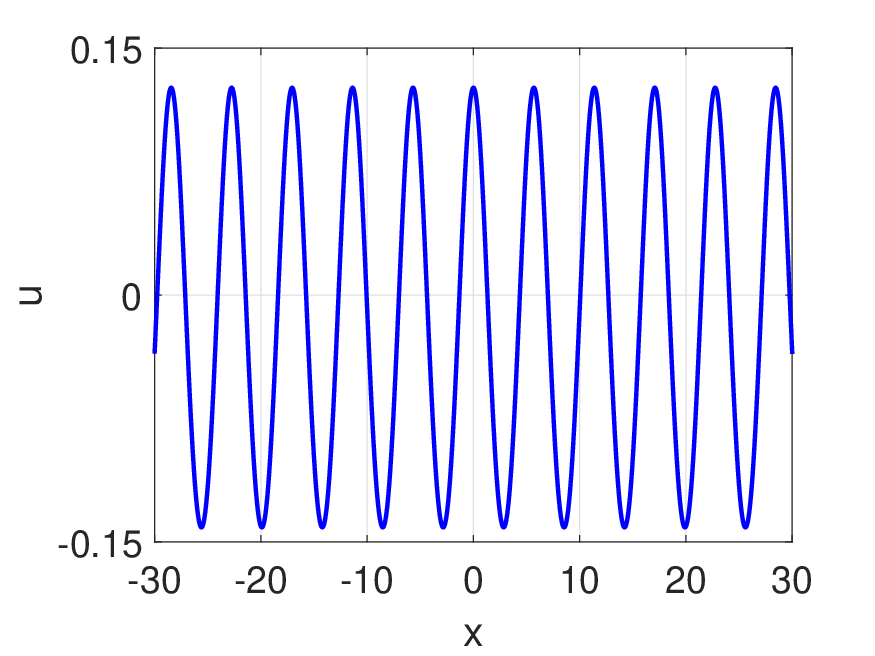}}
  \subfigure[]
  {\includegraphics[width=0.49\textwidth]{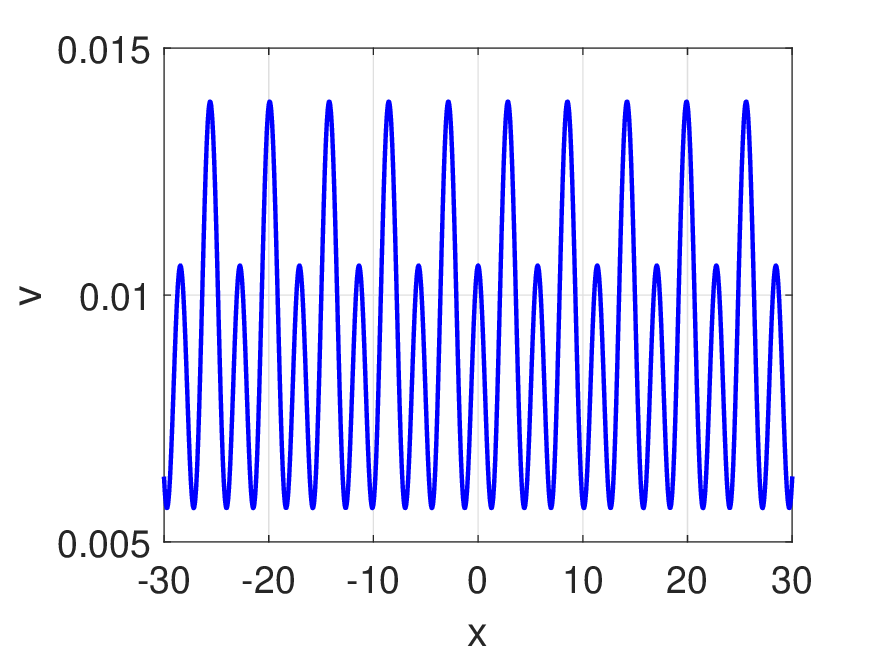}}
  \subfigure[]
  {\includegraphics[width=0.49\textwidth]{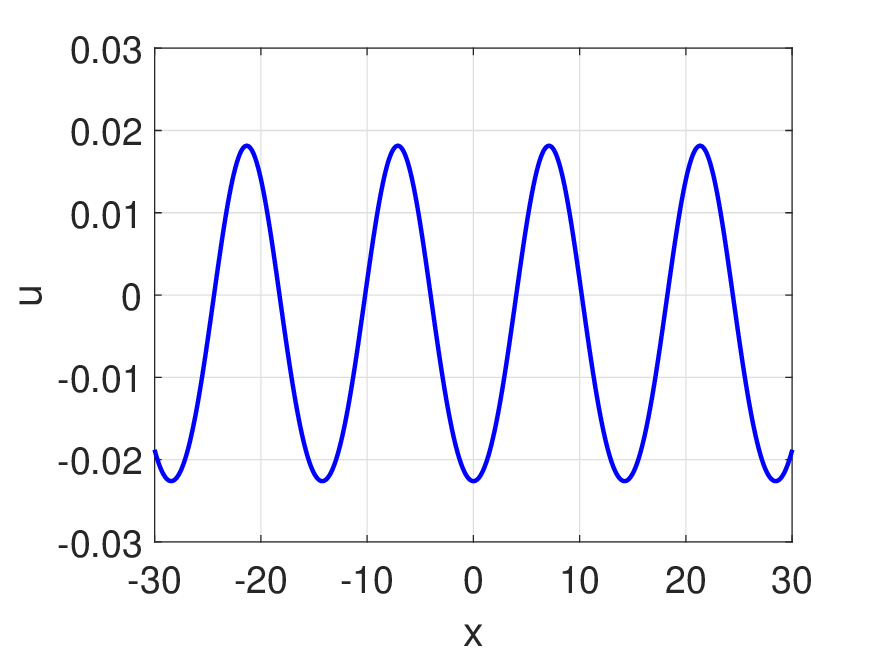}}
  \subfigure[]
  {\includegraphics[width=0.49\textwidth]{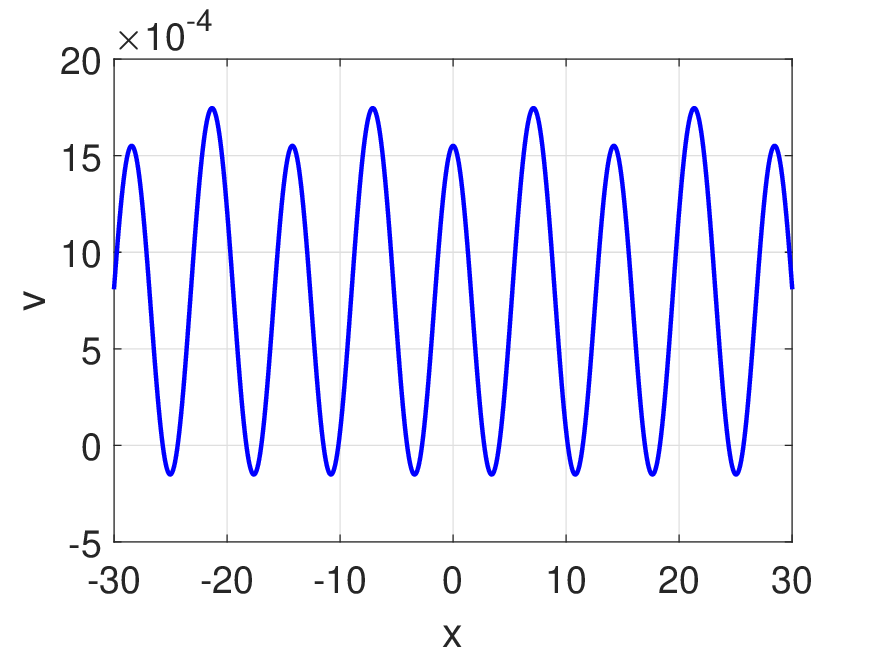}}
  \caption{ Approximate generalized solitary wave profile solutions of (\ref{kgb}) with $f_{1}(u,v)=u^{2}+v^{2}$, $f_{2}(u,v)=u^{2}$: (a),(b) $\alpha=1.2, c_{s}=0.8$ (region 3); (c),(d) $\alpha=1.2, c_{s}=1.1$ (region 4).}
  \label{Fig:gm3}
\end{figure}

\subsection{Existence via Positive Operator theory}
The existence of classical solitary waves can also be justified by using the Positive Operator theory, developed by Benjamin et al. in \cite{BenjaminBB1990} among others. The case of systems of particular interest here can be analyzed from \cite{BonaCh2002}. The theory makes use of the Fourier representation of (\ref{gm2b})
\begin{eqnarray}
((c_{s}^{2}-\alpha^{2})+c_{s}^{2}k^{2})\widehat{u}(k)&=&\widehat{f_{1}(u,v)}(k),\nonumber\\
(1+k^{2}(1-c_{s}^{2}))\widehat{v}(k)&=&\widehat{f_{2}(u,v)}(k),\; k\in\mathbb{R}.\label{pot1}
 \end{eqnarray}
Let us asume that \eqref{gm3d} holds. Then, for all $k\in\mathbb{R}$
\begin{eqnarray*}
p_{1}(k)&=&(c_{s}^{2}-\alpha^{2})+c_{s}^{2}k^{2}>0,\\
p_{2}(k)&=&1+k^{2}(1-c_{s}^{2})>0,
\end{eqnarray*}
and we can invert (\ref{pot1}) to have
\begin{eqnarray*}
\widehat{u}(k)=\frac{\widehat{f_{1}(u,v)}(k)}{p_{1}(k)},\quad
\widehat{v}(k)=\frac{\widehat{f_{2}(u,v)}(k)}{p_{2}(k)},
\end{eqnarray*}
which can be written in a fixed-point form $(u,v)=\mathcal{A}(u,v)$ as
\begin{eqnarray}
u&=&k_{uu}\ast u^{2}+2k_{uv}\ast uv+k_{vv}\ast v^{2},\nonumber\\
v&=&m_{uu}\ast u^{2}+2m_{uv}\ast uv+m_{vv}\ast v^{2},\label{pot2}
\end{eqnarray}
where $\ast$ denotes convolution and
\begin{eqnarray}
k_{\gamma\beta}(x)=\frac{a_{\gamma\beta}}{sc_{s}^{2}}\sqrt{2\pi}e^{-s|x|},\quad
m_{\gamma\beta}(x)=\frac{b_{\gamma\beta}}{r(1-c_{s}^{2})}\sqrt{2\pi}e^{-r|x|},\label{pot3} 
\end{eqnarray}
for $\gamma,\beta=u,v$ and
$$s=\sqrt{\mu}=\sqrt{1-\alpha^{2}/c_{s}^{2}}, r=1/\sqrt{1-c_{s}^{2}}.$$ 
The application of Positive Operator Theory to \eqref{pot2} guarantees the existence of a solution in the cone
\[
\mathbb{K} =\{(f,g)\in X=C\times C: 
  f, g\,\; \text{are nonnegative non-increasing even functions  on}\, [0,\infty)\},
\]
 where $C=C(\mathbb{R})$ is the class of continuous real-valued functions defined on $\mathbb{R}$. The result is a consequence of the following properties (cf. Theorem 3 of \cite{BonaCh2002}):

\begin{itemize}
\item[(S1)] The functions \eqref{pot3}  satisfy:
\begin{itemize}
\item[(i)]
$k_{\gamma\beta}(-x)=k_{\gamma\beta}(x), m_{\gamma\beta}(-x)=m_{\gamma\beta}(x),\; x\in\mathbb{R},\; \gamma,\beta=u,v$. 
\item[(ii)] Assume that $a_{\gamma\beta}, b_{\gamma\beta}\geq 0$. Then:
\begin{itemize}
\item $k_{\gamma\beta},m_{\gamma\beta}\geq 0$, they are monotone decreasing on $(0,\infty)$, and are convex when $x\geq 0$.
\item Either $m_{uv}+k_{uv}$ or both $k_{uu}+m_{uu}, k_{vv}+m_{vv}$ are strictly convex when $x\geq 0$ if
\begin{itemize}
\item $a_{uv}, b_{uv}$ do not vanish at the same time or
\item $a_{uu}, b_{uu}$ and $a_{vv}, b_{vv}$ do not vanish at the same time.
\end{itemize}
\end{itemize}
\end{itemize}
\item[(S2)] It is clear that if $(u,v)\in \mathbb{K}$ is a fixed point of (\ref{pot2}), then $u=0\Rightarrow v=0$ and vice-versa. Furthermore, we have:
\begin{lemma}
Let $\mathcal{A}$ be the fixed-point operator defined by \eqref{pot2}. Then there are only finitely many fixed points of $A$ in the cone $\mathbb{K}$ which are
constant functions if there are only finitely many solutions of the algebraic system
\begin{eqnarray}
a_{uu}X^{2}+2a_{uv}XY+a_{vv}Y^{2}-\left(\frac{c_{s}^{2}-\alpha^{2}}{2}\right)X&=&0,\nonumber\\
b_{uu}X^{2}+2b_{uv}XY+b_{vv}Y^{2}-\frac{1}{2}Y&=&0.\label{317b}
\end{eqnarray}
\end{lemma}
\begin{proof}
Note that if $(u_{0},v_{0})$ is a constant fixed point of \eqref{pot2} in $\mathbb{K}$, then
\begin{eqnarray*}
u_{0}&=&\int_{-\infty}^{\infty}\left(u_{0}^{2}k_{11}(y)+u_{0}v_{0}k_{12}(y)+v_{0}^{2}k_{22}(y)\right)dy,\\
v_{0}&=&\int_{-\infty}^{\infty}\left(u_{0}^{2}m_{11}(y)+u_{0}v_{0}m_{12}(y)+v_{0}^{2}m_{22}(y)\right)dy.
\end{eqnarray*}
Using \eqref{pot3}, this can be written as
\begin{eqnarray*}
u_{0}&=&\int_{-\infty}^{\infty}\frac{u_{0}a_{uu}+2u_{0}v_{0}a_{uv}+v_{0}^{2}a_{vv}}{sc_{s}^{2}}e^{-s|y|}dy=\frac{2}{s}\frac{u_{0}a_{uu}+2u_{0}v_{0}a_{uv}+v_{0}^{2}a_{vv}}{sc_{s}^{2}},\\
v_{0}&=&\int_{-\infty}^{\infty}\frac{u_{0}b_{uu}+2u_{0}v_{0}b_{uv}+v_{0}^{2}b_{vv}}{r(1-c_{s}^{2})}e^{-r|y|}dy=\frac{2}{r}\frac{u_{0}b_{uu}+2u_{0}v_{0}b_{uv}+v_{0}^{2}b_{vv}}{r(1-c_{s}^{2})}.
\end{eqnarray*}
Therefore, $X=u_{0}, Y=v_{0}$ is a solution of \eqref{317b}.
\end{proof}
\item[(S3)] Note that
\begin{eqnarray*}
&&k_{uv}+k_{vv}\neq 0\quad {\rm if}\quad 2a_{uv}+a_{vv}\neq 0,\\
&&m_{uv}+m_{vv}\neq 0\quad {\rm if}\quad 2b_{uv}+b_{vv}\neq 0.
\end{eqnarray*}
Furthermore, for $\gamma,\beta=u,v$ let
\begin{eqnarray*}
\kappa_{\gamma\beta}&=&=\int_{0}^{2}k_{\gamma\beta}(x)dx=\frac{a_{\gamma\beta}}{sc_{s}^{2}}\sqrt{2\pi}\left(\frac{1-e^{-2s}}{s}\right),\\
\mu_{\gamma\beta}&=&=\int_{0}^{2}m_{\gamma\beta}(x)dx=\frac{b_{\gamma\beta}}{r(1-c_{s}^{2})}\sqrt{2\pi}\left(\frac{1-e^{-2r}}{r}\right).
\end{eqnarray*}
Let $a>0$.  If $a_{\gamma\beta}, b_{\gamma\beta}\geq 0$, then $\kappa_{\gamma\beta},\mu_{\gamma\beta}\geq 0$ and the system of inequalities
\begin{eqnarray*}
a+\kappa_{uu}\int_{0}^{1}u^{2}(x)dx+2\kappa_{uv}\int_{0}^{1}u(x)v(x)dx+\kappa_{vv}\int_{0}^{1}v(x)^{2}dx&\leq &\left(\int_{0}^{1}u(x)^{2}dx\right)^{1/2},\\
a+\mu_{uu}\int_{0}^{1}u^{2}(x)dx+2\mu_{uv}\int_{0}^{1}u(x)v(x)dx+\mu_{vv}\int_{0}^{1}v(x)^{2}dx&\leq &\left(\int_{0}^{1}v(x)^{2}dx\right)^{1/2},
\end{eqnarray*}
implies that each term on the left-hand side is bounded and these bounds are only
dependent on the quantities $\kappa_{\gamma\beta}, \mu_{\gamma\beta}$.
\end{itemize}
\begin{theorem}
Under the hypotheses on the coefficients $a_{\gamma\beta}, b_{\gamma\beta}, \gamma,\beta=u,v$, stated above to satisfy conditions (S1)-(S3) and \eqref{gm3d}, the system \eqref{pot2} has a nontrivial solution $(u,v)\in\mathbb{K}$ which lies in $H^{\infty}\times H^{\infty}$.
\end{theorem}
\begin{remark}
The profile $(u,v)$ will correspond to a CSW solution. The difference with the results in section \ref{nft} concerns the condition \eqref{gm3d}. While the analysis made in section \ref{nft} assumes that $c_{s}^{2}$ is close to $\alpha^{2}$ (thus $\mu>0$ is small), in this case that restriction is not necessary.
\end{remark}
\begin{remark}
Explicit formulas of the solitary waves are in general not known. Some can be derived, under certain hypotheses and from specific forms of the profiles. By way of illustration, consider 
\eqref{kgb} in the case $a_{uv}=b_{uu}=b_{vv}=0$, that is
\begin{eqnarray*} 
			u_{tt}&=&\al^2u_{xx}+u_{ttxx}+\left( a_{uu}u^2+a_{vv}v^2 \right)_{xx}, \label{kgbspecial1} \\
			v_{tt}&=&v_{xx}-v+2b_{uv}uv.   \label{kgbspecial2}
\end{eqnarray*}
The corresponding system \eqref{gm2} has the form


\begin{eqnarray}
      (c_{s}^{2}-\alpha^{2})u-c_{s}^{2}u'' &=& a_{uu}u^{2}+a_{vv}v^{2},  \label{system1}\\
     v-(1-c_{s}^{2})v''&=&2b_{uv}uv. \label{system2}
\end{eqnarray}

We now look for the solutions of the form 
\begin{equation}\label{Ap**1}
u=A_1 {\rm sech}^{2}{(b\xi)},\quad v=A_2 {\rm sech}{(b\xi)},
\end{equation}
for some $b, A_{1}, A_{2}$. Inserting \eqref{Ap**1} into \eqref{system1}, \eqref{system2} yields
%
%
%

\begin{eqnarray}
    && 4A_1b^2 c_s^2+(\alpha^2-c_s^2) A_1+a_{vv}A_2^2 = 0, \label{coefficient1} \\
    && A_1 a_{uu}-6b^2 c_s^2 = 0, \label{coefficient2} \\ 
    && (1-c_s^2)b^2-1 = 0, \label{coefficient3} \\
    && b_{uv}A_1- (1-c_s^2) b^2= 0. \label{coefficient4}
\end{eqnarray}
We solve \eqref{coefficient1}-\eqref{coefficient4}. Assuming $c_{s}^{2}<1$, it holds that
\begin{eqnarray}
&&b^2=\displaystyle \frac{1}{1-c_s^2},\quad a_{uu}=6 (b^2-1) b_{uv},\nonumber\\
&&A_1 = \frac{1}{b_{uv}}, \hspace*{20pt} A_2 = \sqrt{\frac{ c_s^2-\alpha^2 - 4 b^2 c_s^2  }{a_{vv}b_{uv}}}, \label{exact2}
\end{eqnarray}
where \eqref{exact2} requires a choice of the parameters in such a way that $\frac{ c_s^2-\alpha^2 - 4 b^2 c_s^2  }{a_{vv}b_{uv}}>0$. The solitary wave solutions are then given by
%
%
\begin{equation}
    u(x, t) = A_1 \sech^2{\big[ b (x - c_s t)\big]}, \hspace*{20pt} v(x, t) = A_2 \sech{\big[ b (x - c_s t)\big]}. \label{exact1}
\end{equation}

\end{remark}
\section{The KdV approximation for the KGB model}\label{kdv_app}
The last point considered in this paper is concerned with the validity of the KdV approximation for the KGB model, \cite{ChongS,BauerCS}. If we make the ansatz
\begin{eqnarray}
\widetilde{u}(x,t)&=&\epsilon^{2}\psi_{u}^{KdV}(x,t)=\epsilon^{2}A(\epsilon(x-\alpha t),\epsilon^{3}\alpha t),\label{A*1}\\
\widetilde{v}(x,t)&=&\epsilon^{2}\psi_{v}^{KdV}(x,t)=0,\label{A*1b}
\end{eqnarray}
for $\epsilon>0$ small, $A$ some smooth function, going to zero at infinity. Inserting \eqref{A*1}, \eqref{A*1b} into \ref{kgb}, the residuals will satisfy
\begin{eqnarray*}
Res_{u}&=&\epsilon^{6}\left(2\alpha^{2}\partial_{XT}A+\partial_{X}^{4}A+a_{uu}\partial_{X}^{2}(A^{2})\right)+O(\epsilon^{8}),\label{A*2}\\
Res_{v}&=&b_{uu}\epsilon^{4}A^{2}(X,T)+O(\epsilon^{6}),\label{A*2b}
\end{eqnarray*}
where $X=\epsilon(x-\alpha t), T=\epsilon^{3}\alpha t$, and we assume $a_{uu}, b_{uu}>0$. The condition $Res_{u}=O(\epsilon^{8})$ determines, after one integration, the KdV equation for $A$, cf. \cite{Schneider-2020}
\begin{equation}\label{A*3}
2\alpha^{2}\partial_{T}A+\partial_{XXX}A+a_{uu}\partial_{X}(A^{2})=0.
\end{equation}
The proof of the corresponding KdV approximation theorem, \cite{ChongS,BauerCS,Schneider-2020}, requires both residuals at the same elevel of error. This can be done, \cite{ChongS}, modifying \eqref{A*1b} in the form
 \begin{eqnarray*}
\widetilde{v}(x,t)=\epsilon^{4}\psi_{v}^{KdV}(x,t)=\epsilon^{4}B_{1}(\epsilon(x-\alpha t),\epsilon^{3}\alpha t)+\epsilon^{6}B_{2}(\epsilon(x-\alpha t),\epsilon^{3}\alpha t).
\end{eqnarray*}
Now taking
\begin{eqnarray*}
B_{1}=b_{uu}A^{2},\quad B_{2}=(1-\alpha^{2})\partial_{X}^{2}B_{1}+2b_{uv}AB_{1},
\end{eqnarray*}
then, after some calculations, it holds that
\begin{eqnarray}
Res_{u}&=&\epsilon^{8}\left(-\alpha^{2}\partial_{TT}A-2\alpha^{2}\partial_{T}\partial_{X}^{4}A+2a_{uv}\partial_{X}^{2}(AB_{1})\right)+O(\epsilon^{10}),\label{A*5}\\
Res_{v}&=&\epsilon^{8}\left(2\alpha^{2}\partial{TX}B_{1}+(1-\alpha^{2})\partial_{X}^{2}B_{2}+2b_{uv}AB_{2}+b_{vv}B_{1}^{2}\right)+O(\epsilon^{10}).\label{A*5b}
\end{eqnarray}
The estimates \eqref{A*5}, \eqref{A*5b} can be used to analyze the error functions defined by $u=\widetilde{u}+\epsilon^{7/2}R_{u}, v=\widetilde{v}+\epsilon^{7/2}R_{v}$. Arguments based on normal form transformations and energy estimates prove the following approximation result (cf. \cite{Schneider-2020} for a sharper result in the case of unstable resosnances, when $\alpha>2$).
\begin{theorem}
( \cite{ChongS,BauerCS}.) Let $A\in C([0, T_{0}],H^{8})$ be a solution of \eqref{A*3}. Then there exists $\epsilon_{0},C>0$ such that for all $\epsilon\in (0,\epsilon_{0})$ there is a solution $(u,v)$ of \eqref{kgb} satisfying
\begin{eqnarray}
\sup_{t\in [0,T_{0}/\epsilon^{3}]}\sup_{x\in\mathbb{R}}|(u,v)(x,t)-(\epsilon^{2}\psi_{u}^{KdV}(x,t),0)|\leq C\epsilon^{7/2}.\label{A*6}
\end{eqnarray}
\end{theorem}
We can illustrate \eqref{A*6} in the case of a KdV approximation $\epsilon^{2}\psi_{u}^{KdV}(x,t)$ given by a soliton solution  $A(X,T)=\widetilde{A}(X-cT), c>0$ of \eqref{A*3}, which satisfies
\begin{eqnarray*}
\frac{\alpha^{2}}{a_{uu}}c\widetilde{A}'-\frac{1}{2a_{uu}}\widetilde{A}'''= \widetilde{A}\widetilde{A}',
\end{eqnarray*}
and has the form, \cite{MChen1998}
\begin{eqnarray*}
\widetilde{A}(\xi)=3a{\rm sech}^{2}\left(\frac{1}{2}\sqrt{\frac{a}{b}}(\xi)\right),\quad a=\frac{\alpha^{2}}{a_{uu}}c, b=\frac{1}{2a_{uu}}, \xi=X-cT.
\end{eqnarray*}
This leads to
\begin{eqnarray}
\epsilon^{2}\psi_{u}^{KdV}(x,t)=\epsilon^{2}\frac{3c\alpha^{2}}{a_{uu}}{\rm sech}^{2}\left(|\alpha|\sqrt{\frac{c}{2}}(\epsilon(x-\alpha t)-c\epsilon^{3}\alpha t)\right).\label{KdVap}
\end{eqnarray}
Taking $c=0.8, \alpha=1, a_{\gamma\beta}=b_{\gamma\beta}=1, \gamma,\beta=u,v$, and the initial conditions
\begin{eqnarray*}
&&u_{0}(x)=\epsilon^{2}\psi_{u}^{KdV}(x,0),\quad u_{1}(x)=\epsilon^{2}\partial_{t}\psi_{u}^{KdV}(x,0),\\
&&v_{0}(x)=v_{1}(x)=0,
\end{eqnarray*}
the corresponding ivp \eqref{kgb}, \eqref{kgb1} was numerically integrated with an efficient, high-order numerical method, \cite{DDS1}, up to a final time $T=1000$ and for several values of $\epsilon$. The resulting numerical solution was compared with $(\epsilon^{2}\psi_{u}^{KdV}(x,t),0)$. The maximum norm (in $x$) for the difference was measured at several times and the comparison is shown (in semilog scale) in Figure \ref{KdVapp_FIG1}.
\begin{figure}
  \centering
  {\includegraphics[width=0.8\textwidth]{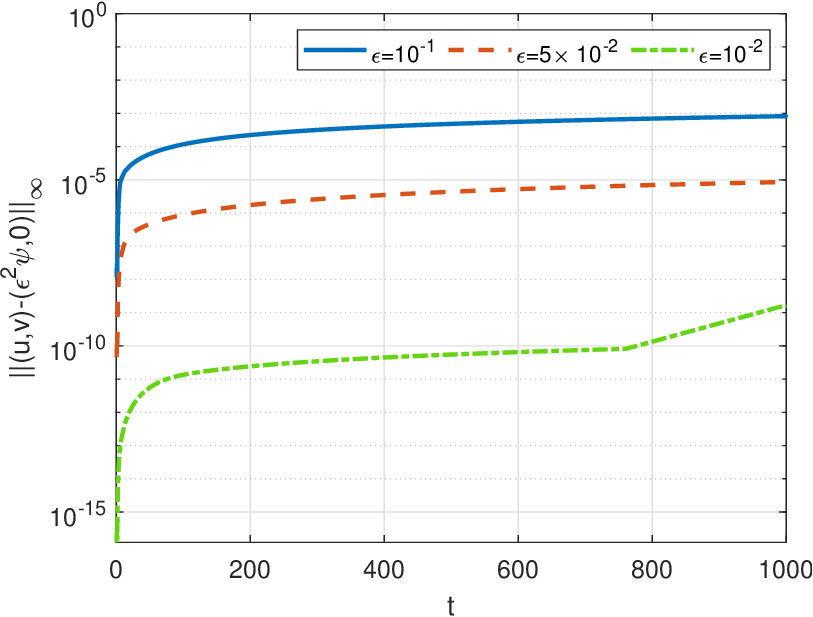}}
  \caption{ Time evolution of the maximum norm of the difference between the numerical solution and $(\epsilon^{2}\psi_{u}^{KdV},0)$ given by \eqref{KdVap} for $c=0.8, \alpha=1, a_{\gamma\beta}=b_{\gamma\beta}=1, \gamma,\beta=u,v$ and several values of $\epsilon$. Semilog scale.}
  \label{KdVapp_FIG1}
\end{figure}
The results suggest that, for the values of $\epsilon$ considered and up to the final time of integration, the errors are $O(\epsilon^{7/2})$ and bounded in time, as established in \eqref{A*6}.  For the case $\epsilon=5\times 10^{-2}$, Figure \ref{KdVapp_FIG2} shows the time behaviour of the numerical approximation to the solution $(u,v)$. Note that the $u$ component seems to evolve, up to the computed final time, as a solitary wave. The preservation of this behaviour for longer times will depend, according to \eqref{A*6} and Figure \ref{KdVapp_FIG1}, on the growth with time of the $O(\epsilon^{7/2})$ remainder terms.

\begin{figure}
  \centering
  \subfigure[]
  {\includegraphics[width=0.49\textwidth]{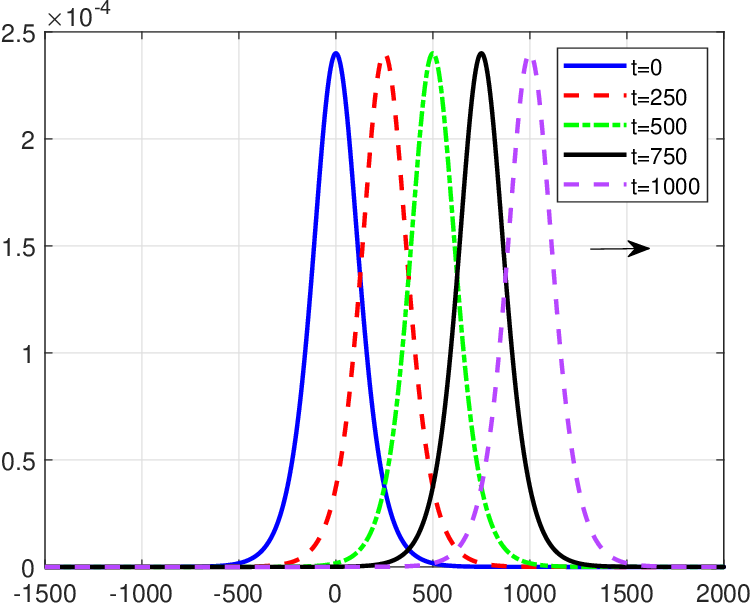}}
  \subfigure[]
  {\includegraphics[width=0.49\textwidth]{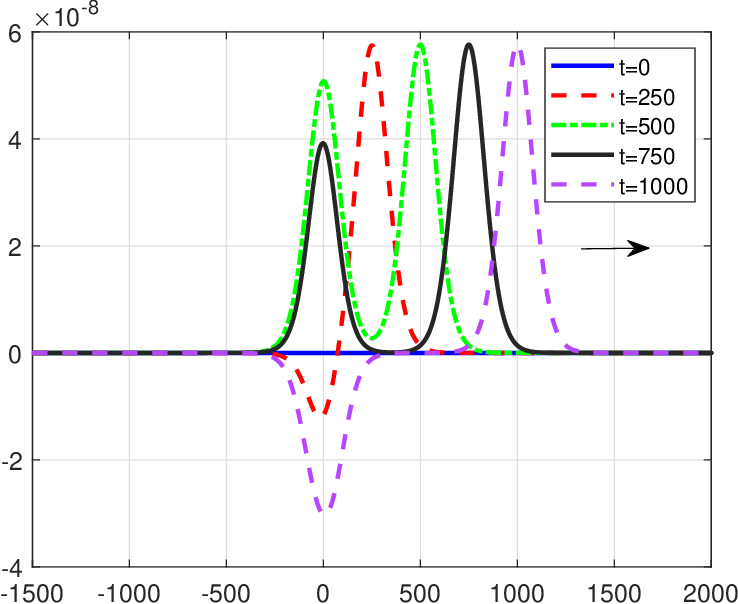}}
  \caption{Time behaviour of the approximate solution of \eqref{kgb}, \eqref{kgb1} with initial data $(\epsilon^{2}\psi_{u}^{KdV},0)$ given by \eqref{KdVap} for $\epsilon=5\times 10^{-2}, c=0.8, \alpha=1, a_{\gamma\beta}=b_{\gamma\beta}=1, \gamma,\beta=u,v$. (a) $u$ component; (b) $v$ component}
  \label{KdVapp_FIG2}
\end{figure}

\section*{Acknowledgments}
 A. E. is supported by the Nazarbayev University under Faculty Development Competitive Research Grants Program  for 2023-2025 (grant number 20122022FD4121). A. D. is supported by the Spanish Agencia Estatal de Investigaci\'on under
Research Grant PID2023-147073NB-I00.

 \section*{Conflict of interests} The author  has no conflicts of interest to declare.
	 
	 \section*{Data availability statement} 
	 Data sharing is not applicable to this article as no new data were created or analyzed in this study.

\appendix

\section{Numerical generation of solitary waves}
\label{appA}

The numerical method to generate approximate solitary wave solutions of \eqref{kgb}, used in section \ref{solit-sec}, is described here.
The system \eqref{gm2} is discretized on
a long enough interval $(-L, L)$ and with periodic boundary conditions by the Fourier
collocation method based on $N$ collocation points $x_j = -L + jh, j = 0, . . . , N-1$ 
for an even integer $N \geq 1$. The vectors
$U = (U_0, . . . , U_{N-1})^T$  and $V = (V_0, . . . , V_{N-1})^T$
 denote, respectively, the
approximations to the values of $u$ and $v$.  The system \eqref{gm2} is implemented in the Fourier space, that is, for
the discrete Fourier components of $U$ and $V$, leading to  $2\times 2$ systems 
\begin{equation}    \label{sys1}
	S(k)
	\left(	\begin{array}{c}
		\hat{U}(k) \\
		
		\hat{V}(k) \\
	\end{array} \right )
	=\left(	\begin{array}{c}
		\widehat{f_1(U,V)}(k) \\
		
		\widehat{f_2(U,V)}(k)\\
	\end{array} \right ),
\end{equation}
where 
\begin{equation*}
    S(k)=\left(	\begin{array}{cc}
		c_s^2-\alpha^2+c_s^2k^2	&  0 \\
		
		0	& 1+(1-c_s^2)k^2 \\
	\end{array} \right ),
\end{equation*}
for  each Fourier component
component  $-\frac{N}{2} \le k \le \frac{N}{2}-1$. 
It can be seen that the matrix  $S(k)$  is nonsingular for $c{s}^{2}\neq \frac{\alpha^{2}}{2k^{2}+1}$. In such case, the iterative resolution of \eqref{sys1} with the classical fixed point  algorithm
\begin{equation*}    \label{sys2}
	\left(	\begin{array}{c}
		\hat{U}^{n+1}(k) \\
		
		\hat{V}^{n+1}(k) \\
	\end{array} \right )=
	\left(	\begin{array}{cc}
		c_s^2-\alpha^2+c_s^2k^2	&  0 \\
		
		0	& 1+(1-c_s^2)k^2 \\
	\end{array} \right )^{-1}
	\left(	\begin{array}{c}
		\widehat{f_1(U^n,V^n)}(k) \\
		
		\widehat{f_2(U^n,V^n)}(k)\\
	\end{array} \right ),\quad n=0,1,\ldots,
\end{equation*}
is typically divergent, \cite{AlvarezD2014}. This can be overcome by inserting a stabilizing factor of the form
\begin{eqnarray} \label{Sn} \nonumber
	{M}_{n}=\frac{\left<	\left(	\begin{array}{cc}
		c_s^2-\alpha^2+c_s^2k^2	&  0 \\
		
		0	& 1+(1-c_s^2)k^2 \\
	\end{array} \right) \left(	\begin{array}{c}
			\hat{U}^{n} \\
			
			\hat{V}^{n} \\
		\end{array} \right ), \left(	\begin{array}{c}
			\hat{U}^{n} \\
			
			\hat{V}^{n} \\
		\end{array} \right ) \right >}
	{\left<	\left(	\begin{array}{c}
			\widehat{f_1(U^n,V^n)}(k) \\
		
		\widehat{f_2(U^n,V^n)}(k)\\
		\end{array} \right ), \left(	\begin{array}{c}
				\hat{U}^{n} \\
			
			\hat{V}^{n} \\
		\end{array} \right ) \right >},\quad n=0,1,\ldots,
	\end{eqnarray}
where $<\cdot,\cdot>$ is the Euclidean inner product in $\mathbb{C}^{2N}$, leading to the Petviashvili method, \cite{Petv1976,pelinovskys}

\begin{equation}   \label{petv} 
	\left(	\begin{array}{c}
		\hat{U}^{n+1}(k) \\
		
		\hat{V}^{n+1}(k) \\
	\end{array} \right )=(M_n)^{2}
	\left(	\begin{array}{cc}
		c_s^2-\alpha^2+c_s^2k^2	&  0 \\
		
		0	& 1+(1-c_s^2)k^2 \\
	\end{array} \right )^{-1}
	\left(	\begin{array}{c}
		\widehat{f_1(U^n,V^n)}(k) \\
		
		\widehat{f_2(U^n,V^n)}(k)\\
	\end{array} \right ),\quad n=0,1,\ldots
\end{equation}
%
%
%
%
%
%
\begin{figure}[htbp]
  \centering
  \subfigure[]
  {\includegraphics[width=0.49\textwidth]{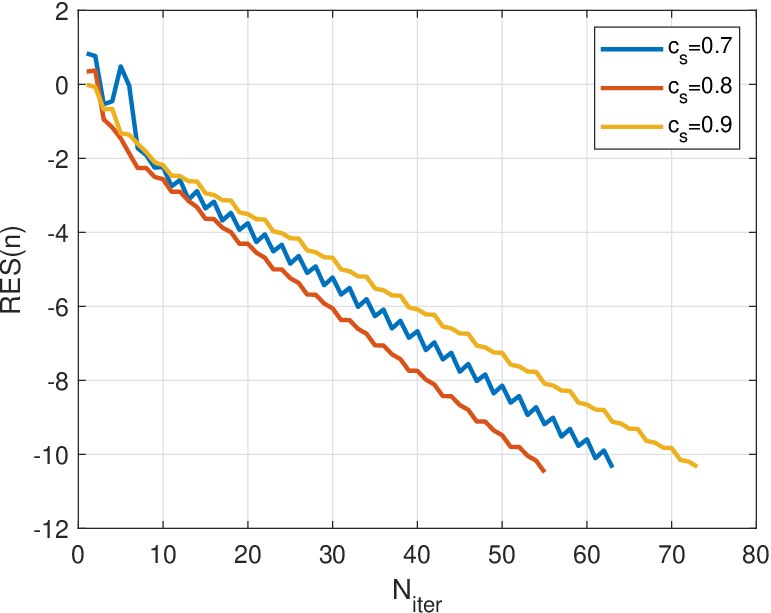}}
  \subfigure[]
  {\includegraphics[width=0.49\textwidth]{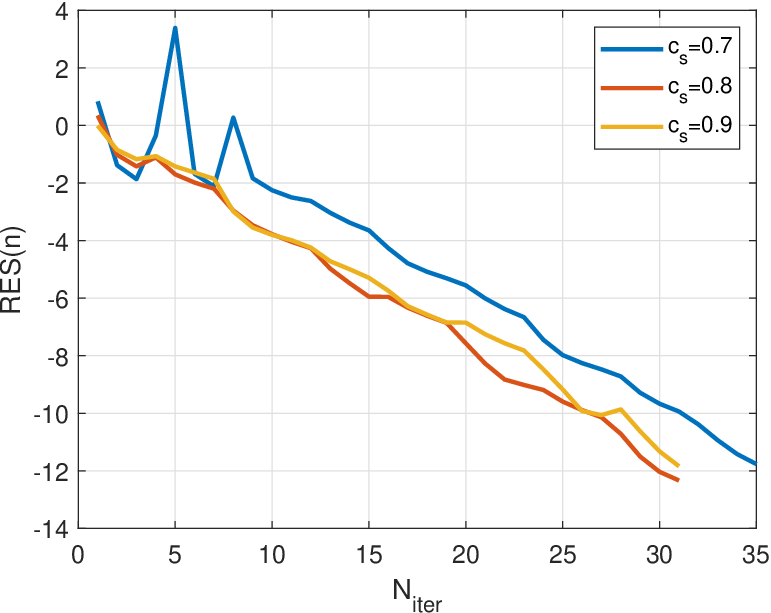}}
  \caption{Residual error generated by (a) \eqref{petv} and (b) \eqref{petv} with extrapolation, for the experiment in Figure \ref{Fig:gm1}.}
 \label{initial}
\end{figure}

The iterative process \eqref{petv} is controlled by the residual error
\begin{equation}\label{resi}
{RES(n)}=  \norm{ \left(	\begin{array}{cc}
		c_s^2-\alpha^2+c_s^2k^2	&  0 \\
		
		0	& 1+(1-c_s^2)k^2 \\
	\end{array} \right) \left(	\begin{array}{c}
			\hat{U}^{n} \\
			
			\hat{V}^{n} \\
		\end{array} \right )
      -\left(	\begin{array}{c}
			\widehat{f_1(U^n,V^n)}(k) \\
		
		\widehat{f_2(U^n,V^n)}(k)\\
		\end{array} \right )
  }_2,\quad n=0,1,\ldots ,
\end{equation}
and it can be  complemented by extrapolation techniques, which may accelerate its convergence,
\cite{sidi1}, \cite{sidi2}, \cite{sidi3}.

The accuracy of the computed profiles is checked in Figure \ref{initial}, which shows the behaviour of the residual error \eqref{resi} as function of the number of iterations, for the waves computed in Figure  \ref{Fig:gm1}. Figure \ref{initial} corresponds to the application of the Petviashvili method \eqref{petv}, while in Figure \ref{initial}(b) the iteration is accelerated with an extrapolation technique.

\end{document}